\documentclass[reqno,10pt]{amsart}
\usepackage{amscd,amssymb}
\usepackage{latexsym}
\usepackage[all]{xy}

\def\into{\hookrightarrow}

\def\Lra{\Leftrightarrow}
\def\toisom{\widetilde{\to}}

\def\.{,\dots ,}
\def\wt{\widetilde}

\def\ol{\overline}

\def\Val{{\rm Val}}

\def\Tor{{\rm Tor}}

\def\Spa{{\rm Spa}}
\def\Spv{{\rm Spv}}
\def\Spec{{\rm Spec}}
\def\Proj{{\rm Proj}}
\def\Frac{{\rm Frac}}
\def\bfSpec{{\bf Spec}}
\def\bfProj{{\bf Proj}}

\def\reg{{\rm reg}}

\def\Ker{{\rm Ker}}

\def\RZ{{\rm RZ}}

\def\pr{{\rm pr}}

\def\bfN{{\bf N}}

\def\bfP{{\bf P}}

\def\bfZ{{\bf Z}}

\def\bfx{{\bf x}}
\def\bfy{{\bf y}}
\def\bfz{{\bf z}}

\def\gtS{{\mathfrak S}}

\def\gtU{{\mathfrak U}}

\def\gtX{{\mathfrak X}}
\def\gtY{{\mathfrak Y}}

\def\calA{{\mathcal A}}

\def\calE{{\mathcal E}}

\def\calI{{\mathcal I}}

\def\calK{{\mathcal K}}
\def\calL{{\mathcal L}}
\def\calM{{\mathcal M}}

\def\calO{{\mathcal O}}

\def\oA{{\ol A}}
\def\oB{{\ol B}}
\def\oC{{\ol C}}
\def\oD{{\ol D}}

\def\oI{{\ol I}}

\def\oR{{\ol R}}
\def\oS{{\ol S}}
\def\oT{{\ol T}}
\def\oU{{\ol U}}

\def\oX{{\ol X}}
\def\oY{{\ol Y}}

\def\of{{\ol f}}

\def\oy{{\ol y}}

\def\tilK{{\wt K}}

\def\tilR{{\wt R}}

\def\tilf{{\wt f}}
\def\tilg{{\wt g}}
\def\tilh{{\wt h}}

\def\tilx{{\wt x}}

\def\ogtX{{\ol\gtX}}

\def\obfy{{\ol\bfy}}

\def\oeta{{\ol\eta}}

\def\ophi{{\ol\phi}}

\def\opsi{{\ol\psi}}

\def\tilphi{{\wt\phi}}

\def\alp{{\alpha}}
\def\lam{{\lambda}}

\def\veps{\varepsilon}
\def\ve{\veps}

\def\R+*{{\bf R^*_+}}

\newtheorem{theor}{Theorem}[subsection]
\newtheorem{prop}[theor]{Proposition}
\newtheorem{lem}[theor]{Lemma}
\newtheorem{cor}[theor]{Corollary}

\theoremstyle{definition}

\newtheorem{defin}[theor]{Definition}
\newtheorem{rem}[theor]{Remark}

\begin{document}

\title{Relative Riemann-Zariski spaces}
\author{Michael Temkin}
\address{\tiny{Einstein Institute of Mathematics, The Hebrew University of Jerusalem, Giv'at Ram, Jerusalem, 91904, Israel}}
\email{\scriptsize{temkin@math.huji.ac.il}}
\thanks{I want to express my deep gratitude to B. Conrad for pointing out various gaps and mistakes in an earlier
version of the article and to thank R. Huber for a useful discussion. Also I thank D. Rydh and the referee for
pointing out some mistakes in \S2.3. A first version of the article was written during my stay at the Max Planck
Institute for Mathematics at Bonn. The final revision was made when the author was staying at IAS and supported
by NFS grant DMS-0635607.}

\begin{abstract}
In this paper we study relative Riemann–Zariski spaces associated to a
morphism of schemes and generalizing the classical Riemann–Zariski space
of a field. We prove that similarly to the classical RZ spaces, the relative
ones can be described either as projective limits of schemes in the category
of locally ringed spaces or as certain spaces of valuations. We apply these
spaces to prove the following two new results: a strong version of stable
modification theorem for relative curves; a decomposition theorem which
asserts that any separated morphism between quasi-compact and quasi-separated
schemes factors as a composition of an affine morphism and
a proper morphism. In particular, we obtain a new proof of Nagata's
compactification theorem.
\end{abstract}

\maketitle


\section{Introduction}

Let $K/k$ be a finitely generated field extension. In the first half of the 20-th century, Zariski defined a
Riemann variety $\RZ_K(k)$ as the projective limit of all projective $k$-models of $K$. Zariski showed that this
topological space, which is now called a Riemann-Zariski (or Zariski-Riemann) space, possesses the following
set-theoretic description: to give a point $\bfx\in\RZ_K$ is equivalent to give a valuation ring $\calO_\bfx$
with fraction field $K$ and such that $k\subset\calO_\bfx$. The Riemann-Zariski space possesses a sheaf of rings
$\calO$ whose stalks are valuation rings of $K$ as above. Zariski made extensive use of these spaces in his
desingularization works.

Let $S$ be a scheme and $U$ be a subset closed under generalizations, for example $U=S_\reg$ is the regular
locus of $S$, or $U=\eta$ is a maximal point of $S$. In many birational problems one wants to consider only
$U$-modifications $S'\to S$, i.e. modifications which do not modify $U$. Then it is natural to consider the
projective limit $\gtS=\RZ_U(S)$ of all $U$-modifications of $S$. It was remarked in \cite[\S3.3]{Temst} that
working with such relative Riemann-Zariski spaces one can extend the $P$-modification results of \cite{Temst} to
the case of general $U$ and $S$, and this plan is realized in \S\ref{applchap}. In \S\ref{affsec} we give a
preliminary description of the space $\gtS$, which is used in \S\ref{applicsec} to prove the first main result
of the paper, the stable modification theorem \ref{stabmodtheorU} generalizing its analog from \cite{Temst}. Our
improvement to the stable modification theorem \cite[1.5]{Temst} is in the control on the base change one has to
perform in order to construct a stable modification of a relative curve $C\to S$. Namely, we prove that in order
to find a stable modification of a relative curve with semi-stable $U$-fibers it suffices to replace the base
$S$ with a $U$-\'etale covering.

Although a very rough study of relative RZ spaces suffices for the proof of Theorem \ref{stabmodtheorU}, it
seems natural to investigate these spaces deeper. Furthermore, the definition of relative Riemann-Zariski spaces
can be naturally generalized to the case of an arbitrary morphism $f:Y\to X$, and the case when $f$ is a
dominant point was already applied in \cite{Tem1}. So, it is natural to investigate the relative RZ spaces
associated to a morphism $f:Y\to X$. We will see that under a very mild assumption that $f$ is a separated
morphism between quasi-compact quasi-separated schemes, one obtains a very specific description of the space
$\RZ_Y(X)$ which is similar to the classical case of $\RZ_K(k)$. Let us say that $f$ is {\em decomposable} if it
factors into a composition of an affine morphism $Y\to Z$ and a proper morphism $Z\to X$. Actually, in
\S\ref{affsec} we study $\RZ_Y(X)$ in the case of a general decomposable morphism because this case is not
essentially easier than the case of an open immersion $Y\into X$. We define a set $\Val_Y(X)$ whose points are
certain $X$-valuations of $Y$, and construct a surjection $\psi:\Val_Y(X)\to\RZ_Y(X)$. It will require some
additional work to prove in Corollary \ref{homeomcor} that $\psi$ is actually a bijection (and even a
homeomorphism with respect to natural topologies defined in the paper). Now, a natural question to ask is if the
decomposition assumption is essential. Slightly surprisingly, the answer is negative because the assumption is
actually empty. A second main result of this paper is decomposition theorem \ref{decompth} which states that a
morphism of quasi-compact quasi-separated schemes is decomposable if and only if it is separated. Thus, the
description of relative RZ spaces obtained in the decomposable case is actually the general one.

We give two proofs of the decomposition theorem in this paper. The first proof is based on Nagata
compactification and Thomason approximation theorems. Actually, we prove in \S\ref{intrsec} that the
decomposition theorem is essentially equivalent to the union of these two theorems. This accomplishes the first
proof. On the other hand, it turns out that a deeper study of relative RZ spaces leads to an independent proof
of the decomposition theorem as explained in \S\ref{mainsec}. In particular, we obtain new proofs of Nagata's
and Thomason's theorems. Though there are few known proofs of Nagata's theorem, see \cite{Con} and \cite{Lut},
the author expects that the new proof might be better suited for applying to algebraic spaces and (perhaps)
certain classes of stacks (joint project with I. Tyomkin).

Let us describe briefly the structure of the paper. In \S\ref{intrsec} we prove a slight generalization of
Thomason's theorem and show that the decomposition theorem is essentially equivalent to the union of Nagata's
and Thomason's theorems. In \S\ref{applchap} we start our study of relative RZ spaces and apply them to the
strong stable modification theorem. Then, \S\ref{nagchap} is devoted to further study of the relative RZ spaces.
In \S\ref{spasec} we establish an interesting connection between Riemann-Zariski spaces and adic spaces of R.
Huber; in particular, we obtain an intrinsic topology on $\Val_Y(X)$. However, it turns out that the notion of
an open subdomain in the spaces $\Val_Y(X)$ is much finer than its analog in the adic spaces. It requires some
work to prove in Theorem \ref{affbaseth} that open subdomains of the form $\Val_{\Spec(B)}(\Spec(A))$ form a
basis for the topology of $\Val_Y(X)$. In \S\ref{blowupsec} we study $Y$-blow ups of $X$, which are analogs of
$U$-admissible or formal blow ups from Raynaud's theory, see \cite{BL}. As a corollary, we prove that
$\psi:\Val_Y(X)\to\RZ_Y(X)$ is a homeomorphism in the decomposable case. Finally, we prove in Theorem
\ref{domth} that any open quasi-compact subset of $\Val_Y(X)$ admits a scheme model of the form $\Val_\oY(\oX)$
with $\oY$ being $\oX$-affine. This result implies the decomposition theorem, and, therefore, leads to a new
proof of Nagata's theorem.

I want to mention that I was motivated by Raynaud's theory in my study of Riemann-Zariski spaces in the
decomposable case, and some basic ideas are taken from \cite{BL}. I give a simple illustration of those ideas in
the proof of the generalized Thomason's theorem.

When this paper was almost finished I was informed about a recent paper \cite{FK} by Fujiwara and Kato, which
contains a survey on a theory of generalized Riemann-Zariski spaces they are developing. The survey announces
many interesting results, including Nagata compactification for algebraic spaces. It is clear that there is a
certain overlap between that theory and the present paper which can be rather large, though it is difficult to
make any conclusion on this subject until the actual proofs are published. The generalized RZ spaces mentioned
in \cite{FK} are exactly the relative RZ spaces of open immersions $Y\into X$ (the same case which is used in
the proof of the stable modification theorem).

Finally, let us discuss the most recent progress that was made during the last year. Nagata compactification for
algebraic spaces was proved independently by Conrad-Lieblich-Olsson in \cite{CLO} (implementing Gabber's
approach) and D. Rydh in \cite{Rydh}. In both cases one reduces this to the scheme case rather than proving it
from scratch. It should also be noted in this context that important particular cases of the latter theorem
(when the algebraic spaces are normal or when the target is a field) were proved much earlier by Raoult, see
\cite{R1} and \cite{R2}.

\subsection{On noetherian approximation and Nagata compactification}
\label{intrsec} For shortness, a filtered projective family of schemes with affine transition morphisms will be
called {\em affine filtered family}. Also, we abbreviate the words "quasi-compact and qua\-si-separated" by the
single "word" qcqs. In \cite[C.9]{TT}, Thomason proved a very useful approximation theorem, which states that
any qcqs scheme $Y$ over a ring $\Lambda$ is isomorphic to a scheme $\projlim Y_\alpha$, where
$\{Y_\alpha\}_\alpha$ is an affine filtered family of $\Lambda$-schemes of finite presentation. Due to the
following lemma, this theorem may be reformulated in a more laconic way as follows: $Y$ is affine over a
$\Lambda$-scheme $Y_0$ of finite presentation.

\begin{lem}
A morphism of qcqs schemes $f:Y\to X$ is affine if and only if $Y\toisom\projlim Y_\alpha$, where
$\{Y_\alpha\}_\alpha$ is a filtered family of $X$-affine finitely presented $X$-schemes.
\end{lem}
\begin{proof}
If $Y\toisom\projlim Y_\alpha$ is as in the lemma then $Y_\alpha=\bfSpec(\calE_\alpha)$ for an $\calO_X$-algebra
$\calE_\alp$, hence $Y=\bfSpec(\calE)$ where $\calE=\injlim\calE_\alpha$. Conversely, suppose that $f$ is
affine. By \cite[6.9.16(iii]{egaI}), $f_*(\calO_Y)\toisom\injlim\calE_\alpha$, where $\{\calE_\alpha\}$ is a
filtered family of finitely presented $\calO_X$-algebras. Hence $Y=\projlim\bfSpec(\calE_\alpha)$.
\end{proof}

We generalize Thomason's theorem below. As a by-product, we obtain a simplified proof of the original theorem.

\begin{theor}
\label{approxtheor} Let $f:Y\to X$ be a (separated) morphism of qcqs schemes. Then $f$ can be factored into a
composition of an affine morphism $Y\to Z$ and a (separated) morphism $Z\to X$ of finite presentation.
\end{theor}
\begin{proof}
Step 1. {\sl Preliminary work.} First we observe that if $f$ is separated and $Y\to Z\to X$ is a factorization
as in the theorem, then $Y$ is the projective limit of schemes $Y_\alp$ which are affine over $Z$ and of finite
presentation. By \cite[C.7]{TT}, already some $Y_\alp$ is separated over $X$, hence replacing $Z$ with $Y_\alp$,
we achieve a factorization with $X$-separated $Z$. This allows us to deal only with the general (not necessarily
separated) case in the sequel.

If $Y$ is affine and $f(Y)$ is contained in an open affine subscheme $X'\subset X$ then the claim is obvious.
So, $Y$ admits a finite covering by open qcqs subschemes $Y_1\. Y_n$ such that the induced morphisms $Y_i\to X$
satisfy the conclusion of the theorem. It suffices to prove that one can decrease the natural number $n$ until
it becomes $1$, and, obviously, it suffices to deal only with the case of $n=2$. Then the schemes $U:=Y_1$ and
$V:=Y_2$ can be represented as $U=\projlim U_\beta$ and $V=\projlim V_\gamma$, where the limits are taken over
$X$-affine filtered families of $X$-schemes of finite presentation.

Step 2. {\sl Affine domination.} By \cite[$\rm IV_3$, 8.2.11]{ega}, for $\beta\ge\beta_0$ and
$\gamma\ge\gamma_0$, the schemes $U_\beta$ and $V_\gamma$ contain open subschemes $U'_\beta$ and $V'_\gamma$,
whose preimages in $U$ and $V$ coincide with $W:=U\cap V$. By \cite[$\rm IV_3$, 8.13.1]{ega}, the morphism $W\to
U'_{\beta_0}$ factors through $V'_\gamma$ for sufficiently large $\gamma$. Replace $\gamma_0$ by $\gamma$. By
the same reason, the morphism $W\to V'_{\gamma_0}$ factors through some $U'_\beta$ and the morphism $W\to
U'_\beta$ factors through some $V'_\gamma$. Let us denote the corresponding morphisms as
$f_{\gamma,\beta}:V'_\gamma\to U'_\beta$, $f_{\beta,\gamma_0}$ and $f_{\gamma_0,\beta_0}$. Now comes an obvious
but critical argument: $f_{\beta,\gamma_0}$ is separated because the composition $f_{\gamma_0,\beta_0}\circ
f_{\beta,\gamma_0}:U'_\beta\to U'_{\beta_0}$ is separated (and even affine); $f_{\gamma,\beta}$ is affine
because its composition with the separated morphism $f_{\beta,\gamma_0}$ is affine. We gather the already
defined objects in the left diagram below. Note that everything is defined over $X$, the horizontal arrows are
open immersions, the vertical arrows are affine morphisms and the indexed schemes are of finite
$X$-presentation.
$$
\xymatrix{ V\ar[d]  & W \ar[d]^{\phi'}\ar@{_{(}->}[l]\ar@{^{(}->}[r] & U\ar[dd]^h & & &
V\ar[d]  & W \ar[d]^{\phi'}\ar@{_{(}->}[l]\ar@{^{(}->}[r] & U\ar[d]^\phi\\
V_\gamma& V'_\gamma \ar[d]\ar@{_{(}->}[l] & & & &
V_\gamma& V'_\gamma \ar[d]^{f_{\gamma,\beta}}\ar@{_{(}->}[l]\ar@{^{(}->}[r] & U_\gamma\ar[d]\\
 & U'_\beta \ar@{^{(}->}[r] & U_\beta & & & & U'_\beta \ar@{^{(}->}[r] & U_\beta}
$$

Step 3. {\sl Affine extension.} The main task of this step is to produce the right diagram from the left one. It
follows from the previous stage that $V'_\gamma=\bfSpec(\calE')$, where $\calE'$ is a finitely presented
$\calO_{U'_\beta}$-algebra. The morphism $\phi':W\to V'_\gamma$ to a $U'_\beta$-affine scheme corresponds to a
homomorphism $\varphi':\calE'\to h'_*(\calO_W)$, where $h':W\to U'_\beta$ is the projection. Obviously
$h_*(\calO_U)|_{U'_\beta}\toisom h'_*(\calO_W)$, where $h:U\to U_\beta$ is the projection. Hence we can apply
\cite[6.9.10.1]{egaI}, to find a finitely presented $\calO_{U_\beta}$-algebra $\calE$ and a homomorphism
$\varphi:\calE\to h_*(\calO_U)$ such that $\calE|_{U'_\beta}\toisom\calE'$ and the restriction of $\varphi$ to
$U'_\beta$ is $\varphi'$. Set $U_\gamma=\bfSpec(\calE)$, then $U_\gamma\to U_\beta$ is an affine morphism whose
restriction over $U'_\beta$ is $f_{\gamma,\beta}$, and $\varphi$ induces a morphism $\phi:U\to U_\gamma$.
Finally, we glue $U_\gamma$ and $V_\gamma$ along $V'_\gamma$ obtaining a finitely presented $X$-scheme $Z$, and
notice that the affine morphisms $U\to U_\gamma$ and $V\to V_\gamma$ glue to an affine morphism $Y\to Z$ over
$X$.
\end{proof}

Our proof is a simple analog of Raynaud's theory. Thomason used the first two steps (induction argument in the
proof of Theorem C.9 and Lemma C.6). Our simplification of his proof is due to the third step. The same
arguments are used in Raynaud's theory, see the end of the proof of \cite[4.1(d)]{BL} and \cite[2.6(a)]{BL}. In
our paper, they also appear in the proofs of Lemmas \ref{blowuplem}(i) and \ref{extblowuplem}, and Theorem
\ref{domth}.

Next, we recall Nagata compactification theorem, see \cite{Nag}. A scheme theoretic proof of the theorem can be
found in \cite{Con} or \cite{Lut}. Recall that a morphism $f:Y\to X$ is called {\em compactifiable} if it can be
factored as a composition of an open immersion $g:Y\to Z$ and a proper morphism $h:Z\to X$. Nagata proved that a
finite type morphism $f:Y\to X$ of qcqs schemes is compactifiable if and only if it is separated. Actually,
Nagata considered noetherian schemes, and the general case was proved by B. Conrad in \cite{Con}.

Assume that $f$ is factored as above. Let $\calI\subset\calO_Z$ be an ideal with support $Z\setminus Y$ and let
$Z'$ be the blow up of $Z$ along $\calI$. We can choose a finitely generated $\calI$ because the morphism
$Y\into Z$ is quasi-compact. The open immersion $g':Y\to Z'$ is affine because $Z'\setminus Y$ is a locally
principal divisor. It follows that $g$ is a composition of an affine morphism $g'$ of finite type and a proper
morphism $Z'\to X$. Conversely, assume that $g:Y\to Z$ is affine of finite type and $Z\to X$ is proper. Then $Y$
is quasi-projective over $Z$, hence there exists an open immersion of finite type $Y\into\oY$ with
$Z$-projective and, therefore, $X$-proper $\oY$. Thus, Nagata's theorem can be reformulated as follows: a finite
type morphism is separated if and only if it can be represented as a composition of an affine morphism of finite
type and a proper morphism. Now, one sees that a weak form of Theorem \ref{approxtheor} ($f$ is separated and
$Z\to X$ is of finite type) and Nagata's theorem are together equivalent to the following decomposition theorem,
which will be also proved in \S\ref{mainsec} by a different method.

\begin{theor}
\label{decompth} A morphism $f:Y\to X$ of quasi-compact quasi-separated schemes is separated if and only if it
can be factored as a composition of an affine morphism $Y\to Z$ and a proper morphism $Z\to X$.
\end{theor}

\section{Preliminary description of relative RZ spaces and applications} \label{applchap}

Throughout \S\ref{applchap}, $f:Y\to X$ denotes a separated morphism between qcqs schemes.

\subsection{Valuations and projective limits}\label{firstsec}

We are going to recall some notions introduced in \cite[\S3.3]{Temst}. Consider a factorization of $f$ into a
composition of a schematically dominant morphism $f_i:Y\to X_i$ and a proper morphism $g_i:X_i\to X$. We call
the pair $(f_i,g_i)$ a {\em $Y$-modification} of $X$, and usually it will be denoted simply as $X_i$. Given two
$Y$-modifications of $X$, we say that $X_j$ {\em dominates} or {\em refines} $X_i$, if there exists an
$X$-morphism $g_{ji}:X_j\to X_i$ compatible with $f_i,f_j,g_i$ and $g_j$. A standard graph argument shows that
if $g_{ji}$ exists then it is unique (one uses only that $f_j$ is schematically dominant and $X_i$ is
$X$-separated). The family $\{X_i\}_{i\in I}$ of all $Y$-modifications of $X$ is filtered because any two
$Y$-modifications $X_i,X_j$ are dominated by the scheme-theoretic image of $Y$ in $X_i\times_X X_j$, and it has
an initial object corresponding to the schematic image of $Y$ in $X$

A relative Riemann-Zariski space $\gtX=\RZ_Y(X)$ is defined as the projective limit of the underlying
topological spaces of $Y$-modifications of $X$. Note that if $X$ is integral and $Y$ is its generic point then
one recovers the classical Riemann-Zariski spaces. A slightly more general case, when $Y$ is a dominant point,
was considered in \cite[\S1]{Tem1}. Let $\pi_i:\gtX\to X_i$ be the projections and $\eta:Y\to\gtX$ be the map
induced by $f_i$'s. We provide $\gtX$ with the sheaf $\calM_\gtX=\eta_*(\calO_Y)$, which will be called the
sheaf of {\em meromorphic functions}, and with the sheaf $\calO_\gtX=\injlim\pi_i^{-1}(\calO_{X_i})$, which will
be called the sheaf of {\em regular functions}. The natural homomorphisms
$\alp_i:\pi_i^{-1}(\calO_{X_i})\to\calM_\gtX$ induce a homomorphism $\alp:\calO_\gtX\to\calM_\gtX$, and we will
prove later that $\eta$ is injective and $\alp$ is a monomorphism. Actually, we will give in Corollary
\ref{lastcor} a rather precise meaning to a claim that $\calM_\gtX$ is a sheaf of semi-fractions of the sheaf
$\calO_\gtX$.

\begin{rem}
For any filtered projective family of locally ringed spaces $\{Y_j\}_{j\in J}$ the projective limit
$\gtY=\projlim_{j\in J}Y_i$ always exists and satisfies $|\gtY|:=\projlim|Y_j|$ and $\calO_\gtY=\injlim
\pi_j^{-1}\calO_{Y_j}$ where $\pi_j:\gtY\to Y_j$'s are the projections. Assume now that $Y_j$'s are schemes.
Then $\gtY$ is known to be a scheme when the transition morphisms are affine: this situation is studied very
extensively in \cite[$\rm IV_3$, \S8]{ega} and the obtained results have a plenty of various very important
applications. Although $\gtY$ does not have to be a scheme in general, it is a locally ringed space of a rather
special form which deserves a study. Our relative RZ spaces $(\gtX,\calO_\gtX)$ provide a nice example of such
pro-schemes (while $\calM_\gtX$ corresponds to an extra-structure related to $Y$), and we will later obtain a
very detailed description of these spaces (e.g. we will describe the stalks of $\calO_\gtX$). Another
interesting example of a pro-scheme which is not a scheme but has a very nice realization is as follows: let $X$
be a scheme with a subset $U$ closed with respect to generalization, then $(U,\calO_X|_U)$ is  the projective
limit of all open neighborhoods of $U$. Note that this locally ringed space does not have to be a scheme: for
example, take $U$ to be the set of all non-closed points on an algebraic surface $X$.
\end{rem}

The classical absolute RZ spaces viewed either as topological spaces or, more generally, as locally ringed
spaces admit two alternative descriptions: (a) a projective limit of schemes, (b) a space whose points are
valuations. We defined the relative spaces $\RZ_Y(X)$ using projective limits, but they also admit a "valuative"
description as spaces $\Val_Y(X)$. In \S\ref{applchap} we only introduce the sets $\Val_Y(X)$ and establish a
certain connection between $\RZ_Y(X)$ and $\Val_Y(X)$ which suffices for application to the stable modification
theorem \ref{stabmodtheorU}. Throughout this paper by a {\em valuation} on a ring $B$ we mean a commutative
ordered group $\Gamma$ with a multiplicative map $|\ |:B\to\Gamma\cup\{0\}$ which satisfies the strong triangle
inequality and sends $1$ to $1$. Recall that if $B$ is a field then $R=\{x\in B|\ |x|\le 1\}$ is a valuation
ring of $B$ (i.e. $\Frac(R)=B$) which defines $|\ |$ up to an equivalence. In general, a valuation is defined up
to an equivalence by its kernel $p$, which is a prime ideal, and by the induced valuation on the residue field
$\Frac(B/p)$. By slight abuse of language, the point of $\Spec(B)$ given by $p$ will be also called the {\em
kernel} of $|\ |$. Also, we will often identify equivalent valuations.

\begin{rem}\label{semivalrem}
We follow R. Huber by using the notion of a valuation. Since these valuations may have a non-empty kernel, a
reasonable alternative, however, would be the notion of a semivaluation. Note also that in the literature on
abstract algebra this object is often called Manis valuation.
\end{rem}

Now, let $\Val_Y(X)$ be the set of triples $\bfy=(y,R,\phi)$, where $y\in Y$ is a point, $R$ is a valuation ring
of $k(y)$ (in particular $\Frac(R)=k(y)$) and $\phi:S=\Spec(R)\to X$ is a morphism compatible with
$y=\Spec(k(y))\to Y$ and such that the induced morphism $y\to S\times_X Y$ is a closed immersion. Let
$\calO_\bfy$ denote the preimage of $R$ in $\calO_{Y,y}$ (currently, it is just a ring attached to $\bfy$). We
would like to axiomatize the properties of $\calO_\bfy$ as follows. By a {\em semi-valuation ring} we mean a
ring $\calO$ with a valuation $|\ |$ such that any zero divisor of $\calO$ lies in the kernel $m=\Ker(|\ |)$ and
for any pair $g,h\in\calO$ with $|g|\le |h|\neq 0$ one has that $h|g$. Two structures of a semi-valuation ring
on $\calO$ are {\em equivalent} if their valuations are equivalent.

Note that $\calO$ embeds into $A=\calO_m$ by our assumption on zero divisors, $mA=m$ because the prime ideal $m$
is $(\calO\setminus m)$-divisible, and $R=\calO/m$ is the valuation ring of $A/m$ corresponding to the valuation
induced by $|\ |$. Therefore, $\calO$ is {\em composed} from the local ring $A$ and the valuation ring $R\subset
A/m$ in the sense that $\calO$ is the preimage of $R$ in $A$. We say that $A$ is a {\em semi-fraction ring} of
$\calO$. Conversely, any ring composed from a local ring and a valuation ring is easily seen to be a
semi-valuation ring. Semi-valuation rings play the same role in the theory of relative RZ spaces as valuation
rings do in the theory of usual RZ spaces.

\begin{rem}\label{lastrem}
(i) The structure of a semi-valuation ring on an abstract local ring $\calO$ is uniquely defined (up to an
equivalence) by its kernel $m$ because $\calO/m$ is a valuation ring and hence defines the valuation. Since
$A=\calO_m$ we obtain that the semi-valuation ring structure on $\calO$ is uniquely defined by its embedding
into the semi-fraction ring $A$.

(ii) An abstract ring $\calO$ can admit many semi-valuation ring structures. For example, if $\calO$ is a
valuation ring then any its localization (i.e. a larger valuation ring in its field of fractions) can serve as
its semi-fraction ring.
\end{rem}

Here is a generalization of the classical criterion that an integral domain $\calO$ is a valuation ring if and
only if for any pair of elements $f,g\in\calO$ either $f|g$ or $g|f$.

\begin{lem}\label{valcritlem}
Let $\calO\subset A$ be two rings. Then the following conditions are equivalent:

(i) $\calO$ admits a structure of a semi-valuation ring such that $A$ is $\calO$-isomorphic to the semi-fraction
ring of $\calO$,

(ii) if $f,g\in A$ are co-prime (i.e. $fA+gA=A$) then either $f\in g\calO$ or $g\in f\calO$.
\end{lem}
\begin{proof}
We should only prove that (ii) implies (i), since the opposite implication is obvious. We claim that $A$ is a
local ring. Indeed, if it is not local then $A\setminus A^\times$ is not an ideal, hence there exist
non-invertible $f,g$ with invertible $f+g$. But by our assumption either $f\in gA$ or $g\in fA$, hence $f+g$ is
contained in a proper ideal equal to either $fA$ or $gA$, that is an absurd. Let $m\subset A$ be the maximal
ideal, then taking $f\in m$ and $g=1$ and observing that $f$ does not divide $1$ in $\calO$ (and even in $A$),
we deduce that $f\in\calO$. Thus, we proved that $m\subset\calO$, in particular, $\calO$ is the preimage of the
ring $\calO/m\subset A/m$ under the surjection $A\to A/m$. It remains to show that $\calO/m$ is a valuation ring
of $A/m$. For a pair of elements $\tilf,\tilg\in \calO/m$ choose liftings $f,g\in\calO$. Since either $f|g$ or
$g|f$ in $\calO$, it follows that either $\tilf|\tilg$ of $\tilg|\tilf$. Hence $\calO/m$ is a valuation ring,
and we are done.
\end{proof}

\subsection{RZ space of a decomposable morphism}
\label{affsec}

Let $\bfy=(y,R,\phi)$ be a point of $\Val_Y(X)$ and let $S=\Spec(R)$. By the valuative criterion of properness,
$\phi$ factors uniquely through a morphism $\phi_i:Y\to X_i$ for any $Y$-modification $X_i\to X$. Since
$S\times_{X_i}Y$ is a closed subscheme of $S\times_X Y$ by $X$-separatedness of $X_i$, we obtain that $\phi_i$
induces a closed immersion $y\to S\times_{X_i}Y$, and, in particular, $(y,R,\phi_i)$ is an element of
$\Val_Y(X_i)$. It follows that the natural map $\Val_Y(X_i)\to\Val_Y(X)$ is a bijection. So, $\RZ_Y(X)$ and
$\Val_Y(X)$ depend on $X$ and $Y$ only up to replacing $X$ with its $Y$-modification.

Now we will construct a map of sets $\psi:\Val_Y(X)\to\RZ_Y(X)$. For any $i\in I$, let $x_i\in X_i$ be the {\em
center} of $R$ on $X_i$, i.e. the image of the closed point of $S$ under $\phi_i$. Then the family of points
$(x_i)$ defines a point $\bfx\in\gtX$ and we obtain a map $\psi$ as above. For any $i$, $x_i$ is a
specialization of $f_i(y)$, hence we obtain a homomorphism
$\calO_{X_i,x_i}\to\calO_{X_i,f_i(y)}\to\calO_{Y,y}\to k(y)$ whose image lies in $R$ because $x_i$ is the center
of $R$ on $X_i$. Therefore, the image of $\calO_{X_i,x_i}$ in $\calO_{Y,y}$ lies in $\calO_\bfy$, and we obtain
a natural homomorphism $\calO_{\gtX,\bfx}=\injlim\calO_{X_i,x_i}\to\calO_\bfy$.

\begin{prop}
\label{goodcaseprop} Suppose that $f$ is decomposable. Then any point $\bfx\in\gtX$ possesses a preimage
$\bfy=\lam(\bfx)$ in $\Val_Y(X)$ such that the homomorphism $\calO_{\gtX,\bfx}\to\calO_\bfy$ is an isomorphism.
In particular, $\lam$ is a section of $\psi$.
\end{prop}
Actually, we will prove in \S\ref{nagchap} that $\psi$ is a bijection (so $\lam$ is its inverse), but the
proposition as it is already covers our applications in \S\ref{applchap}.
\begin{proof}
Factor $f$ into a composition of an affine morphism $Y\to Z$ and a proper morphism $Z\to X$. After replacing $X$
with the scheme-theoretic image of $Y$ in $Z$, we can assume that $f$ is affine. Note that then for any
$Y$-modification $X_i\to X$, the morphism $f_i:Y\to X_i$ is affine. Let $x_i$ be the image of $\bfx$ in $X_i$.
Obviously, the schemes $U_i=\Spec(\calO_{X_i,x_i})\times_{X_i}Y$ are affine. In addition, on the level of sets
each $U_i$ consists of points $y\in Y$ such that $x_i$ is a specialization of $f_i(y)$, the morphisms $U_i\to Y$
are topological embeddings and $\calO_Y|_{U_i}\toisom\calO_{U_i}$. Notice that the schemes $U_i=\Spec(B_i)$ form
a filtered family, hence $U_\infty:=\projlim U_i=\Spec(B_\infty)$, where $B_\infty=\injlim B_i$. By \cite[$\rm
IV_3$, \S8]{ega}, $U_\infty=\cap U_i$ set-theoretically. Since $f_i:Y\to X_i$ is schematically dominant and the
latter property is preserved under (possibly infinite) localizations on the base, the morphism
$U_i\to\Spec(\calO_{X_i,x_i})$ is schematically dominant too. So, for each $i\in I$ we have that
$\calO_{X_i,x_i}\into B_i$, and then an embedding of the direct limits $\calO_{\gtX,\bfx}\into B_\infty$ arises.

\begin{lem}
Suppose that elements $g,h\in B_\infty$ do not have common zeros on $U_\infty$. Then either $g\in
h\calO_{\gtX,\bfx}$ or $h\in g\calO_{\gtX,\bfx}$.
\end{lem}
\begin{proof}
Find $i$ such that $g$ and $h$ are defined and do not have common zeros on $U_i$. Note that $U_i=\cap
f^{-1}(V_j)$, where $V_j$ runs over affine neighborhoods of $x_i$. Hence we can choose a neighborhood
$X'_i=\Spec(A)$ of $x_i$ such that $g,h\in B$ and $gB+hB=1$, where $Y'=\Spec(B)$ is the preimage of $X'_i$ in
$Y$. To ease the notation we will write $X$ and $x$ instead of $X_i$ and $x_i$ (we can freely replace $X$ with
$X_i$ because $\RZ_Y(X)$ remains unchanged). Now, the pair $(g,h)$ induces a morphism $\alp':Y'\to
P':=\Proj(A[T_g,T_h])$, whose scheme-theoretic image $\oX'$ is a $Y'$-modification of $X'$. It would suffice to
extend the $Y'$-modification $\alp':\oX'\to X'$ to a $Y$-modification $\alp:\oX\to X$. Indeed, either $T_g\in
T_h\calO_{\oX',x'}$ or $T_h\in T_g\calO_{\oX',x'}$, where $x'\in\oX'$ is the image of $\bfx$ in $\oX$. So,
existence of $\alp$ would imply that $g|h$ or $h|g$ already in the image of $\calO_{\oX,x'}$ in $B_\infty$,
which is by definition contained in $\calO_{\gtX,\bfx}$.

It can be difficult to extend $\alp'$ (without applying Nagata compactification), but fortunately we can replace
$\oX'$ with any its $Y'$-modification $\oX''$ and it suffices to extend $\oX''\to X'$ to a $Y$-modification of
$X$. Choose $a,b$ such that $ag+bh=1$. Then there exists a natural morphism $\beta':Y'\to
P'':=\Proj(A[T_{ag},T_{ah},T_{bg},T_{bh}])$ which takes $Y'$ to the affine chart on which $T_{ag}+T_{bh}$ is
invertible. We define $\oX''$ to be the scheme-theoretic image of $\beta'$. Since $\beta'$ factors through Segre
embedding $\Proj(A[T_g,T_h])\times\Proj(A[T_a,T_b])\into P''$, we obtain that $\oX''$ is a closed subscheme of
the source which is mapped to $\oX'$ by the projection onto the first factor. In particular, $\oX''$ is a
$Y'$-modification of $\oX'$. We will show that $\oX''\to X'$ extends to a $Y$-modification $\oX\to X$.

Let $E\subset B$ be the $A$-submodule generated by $ag,ah,bg,bh$ and consider the graded algebra
$A_E:=\oplus_{n=0}^\infty E^n$, where $E^n$ is the $n$-th power of $E$ in $B$ and $E^0$ is the image of $A$.
Note that $1\in E$, and we will denote by $1_E$ the associated $1$-graded element of $A_E$. Set $P:=\Proj(A_E)$
and observe that the affine chart corresponding to $1_E$ is $P_1=\Spec(\cup_{n=0}^\infty E^n)$ (where the union
is taken inside of $B$). Clearly, $P$ is a closed subscheme in $P''$ and the morphism $Y\to P''$ factors through
$P_1$. In particular, $\oX''$ is the schematical image of $Y\to P$, and the latter coincides with the
schematical closure of $P_1$ because the morphism $Y\to P_1$ is schematically dominant by injectivity of the
homomorphism $\cup_{n=0}^\infty E^n\to B$. By \cite[6.9.7]{egaI}, $E$ can be extended to a finitely generated
$\calO_X$-submodule $\calE\subset f_*(\calO_Z)$, and replacing $\calE$ by $\calE+\calO_X$ we achieve in addition
that $\calE$ contains the image of $\calO_X$ in $f_*(\calO_Y)$. Let $\calE^n$ be the $n$-th power of $\calE$ in
the sheaf of $\calO_X$-algebras $f_*(\calO_Y)$ (so, $\calE^0$ is the image of $\calO_X$) and form the graded
$\calO_X$-algebra $\calA_\calE:=\oplus_{n=0}^\infty\calE^n$. Then exactly the same computation as was used above
shows that the schematical closure of $\bfSpec(\cup_{n=0}^\infty\calE^n)$ in $\bfProj(\calA_\calE)$ is a
$Y$-modification of $X$, which we denote $\oX$. Since $\oX\to X$ obviously extends $\oX''\to X'$, we are done.
\end{proof}

The above lemma combined with Lemma \ref{valcritlem} provides $\calO_{\gtX,\bfx}$ with a semi-valuation ring
structure such that $B_\infty$ is its semi-fraction ring. In particular, $B_\infty$ is a local ring and so
$U_\infty$ possesses a unique closed point $y$. Thus, $B_\infty=\calO_{Y,y}$, its subring $\calO_{\gtX,\bfx}$
contains $m_y$ and $R:=\calO_{\gtX,\bfx}/m_y$ is a valuation ring of $k(y)$. Define $\phi:S=\Spec(R)\to X$ as
the composition of the closed immersion $S\to\Spec(\calO_{\gtX,\bfx})$ with the natural morphism
$\Spec(\calO_{\gtX,\bfx})\to X$. Since $\calO_{\gtX,\bfx}$ is composed from $\calO_{Y,y}$ and $R$, the triple
$\bfy:=(y,R,\phi)$ is a candidate for being $\lam(\bfx)$ and it only remains to check that $y\to S\times_XY$ is
a closed immersion (and so $\bfy$ is indeed an element of $\Val_Y(X)$).

For any $i$, $U_i=\Spec(\calO_{X_i,x_i})\times_{X_i} Y$ is a closed subscheme of $\Spec(\calO_{X_i,x_i})\times_X
Y$, hence $U_\infty\toisom\projlim_{i\in I}U_i$ is a closed subscheme of $$\Spec(\calO_{\gtX,\bfx})\times_X
Y\toisom\projlim_{i\in I}\Spec(\calO_{X_i,x_i})\times_X Y$$ Since $y$ is closed in $U_\infty$, we obtain that
the morphism $y\to\Spec(\calO_{\gtX,\bfx})\times_X Y$ is a closed immersion. Hence the morphism from $y$ to a
closed subscheme $S\times_X Y$ of $\Spec(\calO_{\gtX,\bfx})\times_X Y$ is a closed immersion too and we are
done.
\end{proof}

\subsection{Applications}
\label{applicsec} A preliminary description of relative Riemann-Zariski spaces obtained in the previous section,
suffices for some applications. Assume we are given a qcqs scheme $S$ with a schematically dense quasi-compact
subset $U$ (i.e. any neighborhood of $U$ is schematically dense) which is closed under generalizations. An
$S$-scheme $X$ is called {\em $U$-admissible} if the preimage of $U$ in $X$ is schematically dense. By a {\em
$U$-\'etale covering} we mean a separated finite type morphism $\phi:S'\to S$ such that $\phi$ is \'etale over
$U$, $S'$ is $U$-admissible, and for any valuation ring $R$ any morphism $\Spec(R)\to S$ taking the generic
point to $U$ lifts to a morphism $\Spec(R')\to S'$ where $R'$ is a valuation ring dominating $R$ and such that
$\Frac(R')/\Frac(R)$ is finite. (Actually those are finite type $h$-covers of $S$ which are \'etale over $U$.)
Note that in \cite{BLR} one considers a more restrictive class of coverings, namely $U$-\'etale maps $S'\to S$,
which split to a composition of a surjective flat $U$-\'etale morphism and a $U$-modification. However, it
follows from the flattening theorem \cite[5.2.2]{RG} of Raynaud-Gruson that the latter class of coverings is
cofinal in ours.

In order to make use of Riemann-Zariski spaces we have first to establish some properties of schemes over
semi-valuation rings. So, let $\calO$ be a semi-valuation ring with semi-fraction ring $A$ and let $m$ be the
maximal ideal of $A$. Recall that $A=\calO_m$, $R:=\calO/m$ is the valuation ring in $K:=A/m$, the scheme
$S=\Spec(\calO)$ is covered by pro-open subscheme $U=\Spec(A)$ (i.e. $U$ is the intersection of open subschemes)
and closed subscheme $T=\Spec(R)$, and the intersection $U\cap T$ is a single point $\eta=\Spec(K)$, which is
the generic point of $T$ and the closed point of $U$. Note that in some sense $S$ is glued from $U$ and $T$
along $\eta$, for example, there is a bi-Cartesian square
$$
\xymatrix{
\eta \ar[d]\ar[r] & U\ar[d]\\
T\ar[r] & S}
$$

Next we will study how $U$-admissible $S$-schemes (resp. quasi-coherent $\calO_S$-modules) can be glued from
$T$-schemes and $U$-schemes (resp. modules), and we will call such gluing $(U,T)$-descent. Given a
quasi-coherent $\calO_S$-module $M$, which we identify with an $\calO$-module, set $M_U=M\otimes_\calO A$,
$M_T=M\otimes_\calO R=M/mM$ and $M_\eta=M\otimes_\calO K$. We say that $M$ is $U$-admissible if the localization
homomorphism $M\to M_U$ is injective. Note that any $\calO$-module $M$ defines a descent datum consisting of
$M_U,M_T$ and an isomorphism $\phi_M:M_U\otimes_AK\toisom M_T\otimes_R K$, and a similar claim holds for
$S$-schemes. The corresponding categories of descent data are defined in an obvious way, and, naturally, we have
a $(U,T)$-descent lemma below. Slightly more generally, we fix a qcqs $U$-admissible $S$-scheme $\oS$ with
$\oU=U\times_S\oS$, $\oT=T\times_S\oS$ and $\oeta=\eta\times_S\oS$ and we will glue objects defined over $\oU$
and $\oT$ along their restrictions over $\oeta$. For example, an $\calO_\oS$-module $\calM$ induces a descent
data $\phi_\calM:\calM_\oU|_\oeta\toisom\calM_\oT|_\oeta$. If we want to stress the choice of $\oS$ we will call
such gluing $(\oU,\oT)$-descent. For an $\oS$-scheme $X$ we will use the notation $X_U=X\times_\oS\oU$ (and so
$X_U\toisom X\times_SU$), $X_T=X\times_\oS\oT$ and $X_\eta=X\times_\oS\oeta$.

\begin{lem} \label{gluelem}
Keep the above notation.

(i) The natural functor from the category of $U$-admissible quasi-coherent $\calO_\oS$-modules (resp.
$\calO_\oS$-algebras) $\calM$ to the category of descent data $(\calM_\oU,\calM_\oT,\phi_\calM)$ with
quasi-coherent $\calO_\oU$-module (resp. $\calO_\oU$-algebra) $\calM_\oU$ and quasi-coherent $\oeta$-admissible
$\calO_\oT$-module (resp. $\calO_\oT$-algebra) $\calM_\oT$ is an equivalence of categories.

(ii) The $(\oU,\oT)$-descent is effective on $\oU$-flat $\oS$-projective schemes with fixed relatively ample
sheaves. More concretely, assume that we are given a descent datum
$((X_U,\calL_U),(X_T,\calL_T),(\phi_X,\phi_\calL))$, where $f_U:X_U\to\oU$ and $f_T:X_T\to\oT$ are projective
morphisms with relatively ample invertible modules $\calL_U$ and $\calL_T$, respectively, $f_U$ is flat, $X_T$
is $\eta$-admissible, $\phi_X:X_U\times_U\eta\toisom X_T\times_T\eta$ and $\phi_\calL$ is an isomorphism between
the restrictions of $\calL_U$ and $\calL_T$ on the $\eta$-fibers which agrees with $\phi_X$. Then there exists a
projective morphism $f:X\to\oS$ with a relatively ample $\calO_X$-module $\calL$ whose restriction over $\oU$
and $\oT$ give rise to the above descent datum.

(iii) A qcqs $U$-admissible $S$-scheme $X$ is of finite type if and only if $X_U$ and $X_T$ are so. If in
addition $X\times_SU\to U$ is flat and finitely presented then $X\to S$ is flat and finitely presentated.
\end{lem}
\begin{proof}
The claim of (i) is local on $\oS$, so we can assume that $\oS$ is affine. Then $\calO_\oS$, $\calO_\oU$ or
$\calO_\oT$-modules can be viewed simply as $\calO$, $A$ or $R$-modules, and this reduces our problem to the
case when $\oS=S$. In particular, we will now denote the modules as $M$, $M_T$, etc. Next we note that $mM_U=mM$
because $mA=m$, and hence $M_T=M/mM$ embeds into $M_\eta=M_U/mM_U$. So, $M_T$ is $\eta$-admissible and the
embedding $M\into M_U$ identifies $M$ with the preimage of $M_T$ under the projection $M_U\to M_\eta$. In
particular, an exact sequence $0\to M\to M_U\oplus M_T\to M_\eta\to 0$ arises. Conversely, given a descent datum
as in (i), we can define an $\calO$-module $M=\Ker(M_U\oplus M_T\to M_\eta)$, and one easily sees that $M$ is
actually the preimage of $M_T\subset M_\eta$ under the projection $M_U\to M_U/mM_U\toisom M_\eta$ and hence
$M_U$ and $M_T$ are the base changes of this $M$. We constructed maps from $\calO_S$-modules to descent data and
vice versa, and one immediately sees that these maps extend to functors. Then it is obvious from the above that
these functors are equivalences of categories which are inverse one to another.

To prove (ii) we find sufficiently large $n$ so that the $n$-th tensor powers of the initial sheaves induce
closed immersions $X_U\to\bfP((f_U)_*(\calL^{\otimes n}_U))$ and $X_T\to\bfP((f_T)_*(\calL^{\otimes n}_T))$ into
the associated projective fibers. Moreover, the higher direct images of $\calL^{\otimes n}_U$ vanish for large
$n$ and then $(f_\eta)_*(\calL^{\otimes n}_\eta)\toisom((f_U)_*(\calL^{\otimes n}_U))_\eta$ by the theorem on
base changes and direct images, see \cite[III.12.9]{Har}. By part (i) the sheaves $(f_U)_*(\calL^{\otimes n}_U)$
and $(f_T)_*(\calL^{\otimes n}_T)$ glue along $(f_\eta)_*(\calL^{\otimes n}_\eta)$ to an $\calO_\oS$-module
$\calM$ and so $\bfP:=\bfP(\calM)$ is glued from $\bfP_U:=\bfP((f_U)_*(\calL^{\otimes n}_U))$ and
$\bfP_T:=\bfP((f_T)_*(\calL^{\otimes n}_T))$ along $\bfP((f_\eta)_*(\calL^{\otimes n}_\eta))$. In particular,
the closed subschemes $X_U\into\bfP_U$ and $X_T\into\bfP_T$ glue to a closed subscheme $i:X\into\bfP$ with a
relatively very ample sheaf $\calK:=i^*(\calO_\bfP(1))$. Note that $\calK$ is glued from $\calL_U^{\otimes n}$
and $\calL_T^{\otimes n}$. Finally, the modules $\calL_U$ and $\calL_T$ glue to an invertible $\calO_X$-sheaf
$\calL$ with $\calL^{\otimes n}\toisom\calK$ by $(X_U,X_T)$-descent of modules, which was established in (i). In
particular, $\calL$ is relatively ample.

The first assertion of (iii) is exactly Step 2 from the proof of \cite[2.5.3]{Temst}. 
So, let us assume that $X\times_\oS\oU\to\oU$ is flat and finitely presented (in addition to the assumption that
$X$ is of finite type over $\oS$ and $\oU$-admissible). The claim is local on $X$ so we can assume that
$X=\Spec(C)$ is affine. Note also that $X_T$ is $\eta$-admissible, so $C/mC$ embeds into $(C/mC)\otimes_RK$ and
then $C/mC$ is flat and finitely presented over $R$ by \cite[3.5.1]{Temst}. We first deal with finite
presentation, so fix an epimorphism $\phi:\calO[T]\to C$ with $T=(T_1\. T_k)$ and let us prove that its kernel
$I$ is finitely generated. Localizing at $m$ we obtain an epimorphism $\phi\otimes_\calO A:A[T]\to B=C_m$ with
kernel $J=I_m$. Then $B$ is $A$-flat by our assumption and we claim that this implies that $J\cap m[T]=mJ$.
Indeed, if $x$ is contained in $J\cap m[T]$ but not in $mJ$ then it reduces to a non-zero element $\tilx$ in the
kernel of $J/mJ\to A[T]/m[T]$. However, this kernel is an epimorphic image of $\Tor_1^{A}(B,m)=0$ and hence
$\tilx=0$. By finite presentation of $A\to B$ we have that $J=\sum_{i=1}^nf_iB$ and multiplying $f_i$'s by
elements of $\calO\setminus m$ we can achieve that $f_i\in I$ and so they generate an ideal $I':=\sum_{i=1}^n
f_iC\subset I$. Since $m$ is $(\calO\setminus m)$-divisible, $mJ=mI'\subset I'$ and hence $\oI:=I\cap
m[T]\subset mJ\subset I'$. Note that $I/\oI$ is the kernel of $\phi\otimes_\calO R:R[T]\to C/mC$, and so is
finitely generated over $R[T]$ (and hence over $\calO[T]$). Choose any finite set of generators $\tilg_1\.
\tilg_l\in I/\oI$, lift each $\tilg_j$ to $g_j\in I$ and consider the ideal $I''=(I',g_1\. g_l)$ in $C[T]$. Then
$I''$ contains $\oI$ and $I''/\oI$ contains $I/\oI$, and so $I=I''$ is finitely generated.

Finally, let us show that $X$ is $S$-flat. We already know that $X$ is of finite presentation over $S$,
therefore the flattening theorem of Raynaud-Gruson \cite[5.2.2]{RG} asserts that $X$ can be flattened by
performing a $U$-modification (and even a $U$-admissible blow up) on $S$ and replacing $X$ with its strict
transform. However, $S$ has no non-trivial $U$-modifications because $T$ (being the spectrum of a valuation
ring) is the only modification of itself. Thus, $X$ has to be $S$-flat and we conclude the proof.
\end{proof}

In the first version of the paper, the lemma was formulated in a larger (and incorrect) generality, as was
pointed out by D. Rydh. So, let us discuss briefly true and false generalizations.

\begin{rem}
(i) Lemma \ref{gluelem}(i) implies that descent data of the form $X_U\times_U\eta\toisom X_T\times_T\eta$ is
always effective for $U$-admissible $\oS$-affine schemes. Some examples show that for general $U$-admissible
schemes the descent of this type is not effective, though it exists as an algebraic space. More generally, Rydh
recently showed in \cite[\S6]{Rydh} that general descent of this type can be made in the category of stacks with
quasi-finite diagonal.

(ii) Also, Rydh observed that the flatness assumption in Lemma \ref{gluelem}(iii) is essential for finite
presentation. Without flatness finite presentation can be lost after gluing even in the case when $\calO$ is a
height two valuation ring composed from DVR's $A$ and $R$ and $X$ is a (non-reduced) closed subscheme in $S$.
\end{rem}

We assume again that $S$ is a qcqs scheme with a schematically dense quasi-compact subset $U$ which is closed
under generalizations. We will prove a stable modification theorem which strengthens its analog from
\cite{Temst}, and we refer to the introduction of loc.cit. for terminology. Our strengthening is in imposing
natural restrictions on the base change required in order to construct a stable modification. It is reasonable
to expect that in some sense one can preserve the locus $U$ of $S$ over which the given curve is already
semi-stable. Since already when $U$ is the generic point of an integral base scheme $S$ one has to allow its
finite \'etale coverings (i.e. one has to allow separable alterations rather then modifications), it seems that
one cannot hope for something more restrictive than admitting general $U$-\'etale coverings of the base.

\begin{theor}
\label{stabmodtheorU} Let $(C,D)$ be an $S$-multipointed curve with semi-stable $U$-fibers. Then there exists a
$U$-\'etale covering $S'\to S$ such that the curve $(C,D)\times_S S'$ admits a stable $U$-modification.
\end{theor}
\begin{proof}
Step 1. {\it The theorem holds over a semi-valuation ring $\calO$.} More concretely, throughout Step 1 we assume
that $\calO$ is composed from a local ring $(A,m)$ and a valuation ring $R$ of $K=A/m$, $S=\Spec(\calO)$ and
$U=\Spec(A)$. Set also $T=\Spec(R)$ and $\eta=\Spec(K)$. By \cite[1.5]{Temst}, the theorem is known in the case
of a valuation ring, i.e. the case when $m=0$. Thus, there exists a finite separable extension $K'/K$ with a
valuation ring $R'$ lying over $R$ and such that $(C,D)\times_S T'$ admits a stable modification, where
$T'=\Spec(R')$. Lift the extension $K'/K$ to a finite \'etale extension of local rings $A'/A$, and let $\calO'$
be the semi-valuation ring composed from $A'$ and $R'$. We will show that the stable modification exists over
$\calO'$, but let us explain first how this concludes the Step. Clearly, $R'=\cup R_i$ where $R_i$'s are
finitely generated $R$-subalgebras of $R'$ such that $\Frac(R_i)=K'$. Therefore $\calO'=\cup \calO_i$ where
$\calO_i$ is the preimage of $R_i$ in $A'$. It remains to note that for any $i$ we have that
$\Spec(\calO_i)\times_SU\toisom\Spec(A')$ is \'etale over $U$, and by approximation the stable modification
exists already over some $\calO_i$.

Now we can work over $\calO'$ and to simplify the notation we replace $\calO$ by $\calO'$ achieving that already
$(C_T,D_T):=(C,D)\times_S T$ admits a stable modification $(\oC_T,\oD_T)$. By \cite[1.1]{Temst} there is a
canonical $C_T$-ample sheaf on $\oC_T$, namely the sheaf $\calL_T:=\omega_{(\oC_T,\oD_T)/T}$. Set also
$\calL_U:=\omega_{(C_U,D_U)/U}$ and note that these sheaves agree over $\eta$ because the formation of
$\omega$'s commutes with base changes (see \cite[\S1]{Temst}). By Lemma \ref{gluelem}(ii) applied with $\oS=C$,
we can glue $(C_U,\calL_U)$ and $(\oC_T,\calL_T)$ to a $U$-modification $\oC\to C$. In addition, $\oC$ is flat
and finitely presented over $S$ by Lemma \ref{gluelem}(iii). Clearly, the closed subschemes $D_U\into C_U$ and
$\oD_T\into\oC_T$ glue to a closed subscheme $\oD\into\oC$, and checking the $S$-fibers we obtain that
$(\oC,\oD)$ is a stable $U$-modification of $(C,D)$.

Step 2. {\it The general case.} Since $(C,D)$ is semi-stable over an open subscheme of $S$, we can enlarge $U$
to an open schematically dense qcqs subscheme. Note that by noetherian approximation there exists a scheme $S'$
of finite type over $\bfZ$ with a morphism $S\to S'$ such that $U$ and $(C,D)$ are induced from a schematically
dense open subscheme $U'\into S'$ and a multipointed curve $(C',D')\to S'$. Then it suffices to solve our
problem for $S',U'$ and $(C',D')$, so we can assume that $S$ is of finite type over $\bfZ$. By
\cite[3.3.1]{Temst}, $\gtS=\RZ_U(S)$ is a qcqs topological space. For any point $\bfx=(y,R,\phi)\in\gtS$, set
$S_\bfx=\Spec(\calO_{\gtS,\bfx})$, $U_\bfx=\Spec(\calO_{Y,y})$ and $(C_\bfx,D_\bfx)=(C,D)\times_S S_\bfx$. Since
the embedding $U\into S$ is obviously decomposable, Proposition \ref{goodcaseprop} implies that
$\calO_{\gtS,\bfx}$ is a semi-valuation ring with the semi-fraction ring $\calO_{Y,y}$. By Step 1, there exists
a $U_\bfx$-\'etale covering $S'_\bfx\to S_\bfx$ such that the $S'_\bfx$-multipointed curve
$(C_\bfx,D_\bfx)\times_{S_\bfx}S'_\bfx\toisom (C,D)\times_S S'_\bfx$ admits a stable $U'_\bfx$-modification for
$U'_\bfx=U_\bfx\times_{S_\bfx}S'_\bfx$. Note also that the morphism $S'_\bfx\to S_\bfx$ is flat and finitely
presented by Lemma \ref{gluelem}(iii).

Consider the family $\{S_i\}_{i\in I}$ of all $U$-modifications of $S$, and let $x_i$ be the center of $\bfx$ on
$S_i$. Recall that $\calO_{\gtS,\bfx}=\injlim\calO_{S_i,x_i}$. By approximation, there exists $i=i(\bfx)$ and a
flat finitely presented $U$-\'etale morphism $h_\bfx:S'\to S_i$ such that $x_i$ lies in its image and
$(C,D)\times_S S'_i$ admits a stable $U$-modification. By flatness of $h_\bfx$, $h_\bfx(S')$ is open in $S_i$,
and hence its preimage in $\gtS$ is an open neighborhood of $\bfx$. Note that in the sequel we can replace $i$
by any larger index $k$ simply by replacing $h_\bfx$ by its base change with respect to the $U$-modification
$S_k\to S_i$. Since $\gtS$ is quasi-compact, there exist finitely many points $\bfx_j$, $1\le j\le n$ with
associated flat morphisms $h_j:S'_j\to S_{i_j}$ so that $\gtS$ is covered by the preimages of the sets
$h_j(S'_j)$. By the above argument we can enlarge all indexes so that $i:=i_1=\dots =i_n$. The open subschemes
$h_j(S'_j)\into S_i$ with $1\le j\le n$ cover $S_i$ because their preimages cover $\gtS$, and so
$S':=\sqcup_{j=1}^n S'_j$ is a flat cover of $S_i$. In particular, $S'$ is a $U$-\'etale covering of $S$ over
which $(C,D)$ possesses a stable $U$-modification.
\end{proof}

A scheme version of the reduced fiber theorem of Bosch-L\"utkebohmert-Raynaud \cite[2.1']{BLR}, can be proved
absolutely similarly.

\begin{theor}
\label{redfibthU} Let $X\to S$ be a schematically dominant finitely presented morphism whose $U$-fibers are
geometrically reduced. Then there exists a $U$-\'etale covering $S'\to S$ and a finite $U$-modification $X'\to
X\times_S S'$ such that $X'$ is flat, finitely presented and has reduced geometric fibers over $S'$.
\end{theor}
\begin{proof}
If $S$ is the spectrum of a valuation ring and $U$ is its generic point then the theorem follows from
\cite[3.5.5]{Temst} (actually it was the content of Steps 2--4 of the loc.cit.). Acting as in Step 1 of the
previous proof, we deduce the case when $S$ is the spectrum of a semi-valuation ring and $U$ is the
corresponding local scheme. Then it remains to repeat the argument of Step 2.
\end{proof}

\section{Relative RZ spaces and the decomposition theorem}
\label{nagchap} Throughout \S\ref{nagchap}, $f:Y\to X$ is a morphism of schemes and $\gtX=\Val_Y(X)$. Later we
will also introduce a topological space $\Spa(Y,X)$ and then we will use the notation $\ogtX=\Spa(Y,X)$.
Sometimes we will consider another morphism of schemes $f':Y'\to X'$ and then $\gtX'=\Val_{Y'}(X')$,
$\ogtX'=\Spa(Y',X')$.

\subsection{Connection to adic spaces}
\label{spasec} Let $A$ be a ring and $B$ be an $A$-algebra. R. Huber considers in \cite{Hub1} the set $\Spv(B)$
of all equivalence classes of valuations on $B$ and provides it with the weakest topology in which the sets of
the form $\{|\ |\in\Spv(B)|\ |a|\le |b|\neq 0\}$ are open for any $a,b\in B$. Huber proves in \cite[2.2]{Hub1}
that the resulting topological space is quasi-compact. Furthermore, he considers the quasi-compact subspace
$\Spa(B,A)\subset\Spv(B)$ consisting of the valuations of $B$ with $|A|\le 1$: see the definition on p. 467 in
loc.cit., where one treats $A$ and $B$ as topological rings with discrete topology (note also that Huber
actually considers the case when $A$ is an integrally closed subring of $B$, but this does not really restrict
the generality because replacing $A$ by the integral closure of its image in $B$ has no impact on the
topological space $\Spa(B,A)$). Actually, the topological space $\Spa(B,A)$ has a much finer structure of an
adic space but we will not use it.

Let us generalize the above paragraph to schemes. Note that a valuation on a ring $A$ is defined by its kernel
$x\in\Spec(A)$ and the induced valuation on $k(x)$. So, by a {\em valuation on} a scheme $Y$ we mean a pair
$\bfy=(y,R)$, where $y\in Y$ is a point called the {\em kernel} of $\bfy$ and $R$ is a valuation ring of $k(y)$.
One can define $\bfy$ by giving a valuation $|\ |_\bfy:\calO_{Y,y}\to\Gamma_\bfy$ whose kernel is $m_y$. By
$\calO_\bfy$ we denote the subring of $\calO_{Y,y}$ given by the condition $|\ |_\bfy\le 1$; it is the preimage
of $R$ in $\calO_{Y,y}$. Remark that $\calO_\bfy$ is a semi-valuation ring with the semi-fraction ring
$\calO_{Y,y}$. Often it is convenient to describe a valuation locally by choosing an affine neighborhood
$\Spec(A)$ of $y$ and giving a valuation $A\to\calO_{Y,y}\to\Gamma_\bfy$ on $A$.

Furthermore, if $f:Y\to X$ is a morphism of schemes then by an {\em $X$-valuation on $Y$} me mean a valuation
$\bfy=(y,R)$ provided with a morphism $\phi:S=\Spec(R)\to X$ which is compatible with the natural morphism
$\eta=\Spec(k(y))\to X$. Recall that in the valuative criteria of properness/separatedness one considers
commutative diagrams of the form
\begin{equation}
\label{valdiag} \xymatrix{
\eta \ar[d]\ar[r]^-{i} & Y\ar[d]^-{f}\\
S\ar[r]^-{\phi} & X}
\end{equation}
where $S=\Spec(R)$ is the spectrum of a valuation ring and $\eta=\Spec(K)$ is its generic point, and studies
liftings of $S$ to $Y$. It is easy to see (and will be proved in Lemma \ref{valdiaglem1}) that it suffices to
consider only the case when $k(y)\toisom K$ for $y=i(\eta)$ in the valuative criteria. In the latter particular
case, diagrams of type (\ref{valdiag}) are exactly the diagrams which correspond to $X$-valuations of $Y$. Note
also that an $X$-valuation $\bfy=(y,R,\phi)$ gives rise to the following finer diagram
\begin{equation}
\label{twosqdiag} \xymatrix{ \eta\ar[r]\ar[d]& \Spec(\calO_{Y,y})\ar[r]\ar[d]& Y\ar[d]\\ S\ar[r]&
\Spec(\calO_\bfy)\ar[r]& X}
\end{equation}
Indeed, the center $x\in X$ of $R$ is a specialization the image of $y$ and the induced homomorphism
$\calO_{X,x}\to\calO_{Y,y}\to k(y)$ coincides with $\calO_{X,x}\to R\to \Frac(R)\toisom k(y)$. Hence these
homomorphisms factor through $\calO_\bfy$ (actually, we have just shown that the left square is co-Cartesian).

Let $\Spa(Y,X)$ denote the set of all isomorphism classes of $X$-valuations on $Y$. We claim that $\Spa(Y,X)$
depends functorially on $f$. Indeed, given a morphism $f':Y'\to X'$ and a morphism $g:f'\to f$ consisting of a
compatible pair of morphisms $g_Y:Y'\to Y$ and $g_X:X'\to X$, there is a natural map
$\Spa(g):\Spa(Y',X')\to\Spa(Y,X)$ which to a point $(y',R',\phi')$ associates a point $(y,R,\phi)$, where
$y=g_Y(y')$, $R=R'\cap k(y)$ and $\phi$ is defined as follows. The morphism $g_X\circ\phi':\Spec(R')\to X$
factors through $\Spec(\calO_{X,x})$, where $x$ is the image of the closed point of the source, hence we obtain
a homomorphism $\alpha:\calO_{X,x}\to R'$. Since the morphism $\Spec(k(y'))\to X$ factors uniquely through
$\Spec(k(y))$, the image of $\alp$ is contained in $R$. So, $g_X\circ\phi'$ factors uniquely through a morphism
$\phi:\Spec(R)\to X$ and the map $\Spa(g)$ is constructed.

If $g_Y$ is an immersion and $g_X$ is separated then $\Spa(g)$ is injective. Indeed, if a point
$\bfy=(y,R,\phi)\in\ogtX:=\Spa(Y,X)$ has a non-empty preimage in $\ogtX':=\Spa(Y',X')$, then $y\in Y'$ and any
preimage of $\bfy$ is given by a lifting of $\phi:\Spec(R)\to X$ to $X'$, which is unique by the valuative
criterion of separatedness. Furthermore, we say that $\ogtX'$ is an {\em affine subset} of $\ogtX$ if $Y'$ and
$X'$ are affine, $g_Y$ is an open immersion and $g_X$ is of finite type. We provide $\ogtX$ with the weakest
topology in which all affine subsets are open. Note that if we are given another morphism between morphisms
$h:(Y_1\to X_1)\to(Y\to X)$ with the corresponding map $\Spa(h):\ogtX_1\to\ogtX$, then $Y'_1:=Y'\times_Y Y_1$ is
a subscheme in $Y_1$ and $X'_1:=X'\times_X X_1$ is separated over $X_1$, hence $\ogtX'_1:=\Spa(Y'_1,X'_1)$
embeds into $\ogtX_1$.

\begin{lem}\label{afflem}
Let $\Spa(g):\ogtX'\to\ogtX$ and $\Spa(h):\ogtX_1\to\ogtX$ be as above and assume that $g_Y$ is an immersion and
$g_X$ is separated.

(i) $\ogtX'_1$ is the preimage of $\ogtX'$ under $\Spa(h)$.

(ii) If $X$ and $Y$ are separated, $\ogtX'$ is an affine subset of $\ogtX$ and both $X_1$ and $Y_1$ are affine,
then $\ogtX'_1$ is an affine subset of $\ogtX_1$. In particular, if $X$ and $Y$ are separated then the
intersection of affine subsets in $\ogtX$ is an affine subset.

(iii) Affine subsets form a basis of the topology on $\ogtX$, and if $X$ and $Y$ are qcqs then any intersection
of two affine subsets is a finite union of affine subsets.

(iv) If $g_Y$ is an open immersion and $g_X$ is of finite type then $\ogtX'$ is open in $\ogtX$;

(v) The maps $\Spa(h)$ are continuous.
\end{lem}
\begin{proof}
The first claim is proved by a straightforward check. If $Y$ and $X$ are separated then $Y'\times_Y Y_1$ and
$X'\times_X X_1$ are affine, hence (i) implies (ii). Furthermore, in general (i) implies that the intersection
of affine subsets in $\ogtX$ is of the form $\Spa(\oY,\oX)$. Since an affine subset in $\Spa(\oY,\oX)$ is also
an affine subset in $\ogtX$, to prove (iii) it suffices to show that any space $\Spa(\oY,\oX)$ (resp. with qcqs
$\oX$ and $\oY$) is covered by (resp. finitely many) affine subsets. Find open affine (resp. finite) coverings
$\oX=\cup\oX_i$ and $\oY=\cup\oY_j$ such that each $\oY_j$ is mapped to some $\oX_{i(j)}$, and note that
$\Spec(\oY,\oX)$ is the union of affine subsets $\Spa(\oY_j,\oX_{i(j)})$. This proves (iii), and the same
argument proves (iv). Finally, (v) follows from the fact the preimage of each affine subset is open due to (i)
and (iv).
\end{proof}

We claim that in the affine case the above topology agrees with the topology defined by Huber.

\begin{lem}
\label{spahomlem} If $X=\Spec(A)$ and $Y=\Spec(B)$ are affine then the canonical bijection
$\phi:\Spa(Y,X)\to\Spa(B,A)$ is a homeomorphism.
\end{lem}
\begin{proof}
It follows from the definitions of the topologies that the map is continuous, so we have only to establish
openness. Let $\ogtX'=\Spa(\Spec(C),\Spec(A'))$ be an affine subset in $\ogtX$, where $\Spec(C)$ is an open
subscheme of $\Spec(B)$ and $A'$ is a finitely generated $A$-algebra. It suffices to prove that $\phi(\ogtX')$
is a neighborhood of each point $\bfz$ it contains. Replacing $A'$ with its image in $C$ we can assume that it
is an $A$-subalgebra of $C$ generated by $h_1\. h_n\in C$. Note that if $\{U_i\}$ is an open covering of
$\Spec(C)$ then the sets $\Spa(U_i,\Spec(A'))$ cover $\ogtX'$. Therefore, shrinking $\Spec(C)$ we can assume
that $C=B_b$ for an element $b\in B$. Then $h_i=b_i/b^m$ with $b_i\in B$ and $m\in\bfN$, and $\phi(\ogtX')$
consists of all valuations of $B$ with $|b_i|\le|b^m|\neq 0$ for any $i$. Thus, $\phi(\ogtX')$ is open in
$\Spa(B,A)$, and we are done.
\end{proof}

Since Huber's spaces $\Spa(B,A)$ are qcqs, we obtain the following corollary.

\begin{cor}\label{qcqscor}
If $X$ and $Y$ are qcqs schemes then the space $\Spa(Y,X)$ is qcqs.
\end{cor}

Let $B$ be a ring provided with a valuation $|\ |:B\to\Gamma\cup\{0\}$, and let $y\in\Spec(B)$ be its kernel. We
say that a convex subgroup $\Gamma'\subseteq\Gamma$ {\em bounds} $B$, if for any element $b\in B$, there exists
an element $h\in\Gamma'$ with $|b|\le h$. For any such subgroup we can define a valuation $|\ |':B\to\Gamma'$ by
the rule $|x|'=|x|$ if $|x|\in\Gamma'$ and $|x|'=0$ otherwise. Obviously, the kernel $y'$ of $|\ |'$ isa
specialization of $y$. Recall that $|\ |'$ is called a {\em primary specialization} of $|\ |$, see
\cite[2.3]{Hub1}. Here are simple properties of primary specializations.

\begin{rem}\label{primrem}
(i) Primary specialization is a transitive operation and the set $P$ of primary specializations of $|\ |$ is
ordered.

(ii) The set $P$ possesses a minimal element corresponding to the intersection of all subgroups bounding $B$; it
is called the {\em minimal primary specialization}.

(iii) A valuation on $B$ is called {\em minimal} if it has no non-trivial primary specializations. For a
valuation given by a point $y\in\Spec(B)$ and a valuation ring $R\subset k(y)$ the following conditions are
equivalent: (a) $(y,R)$ is minimal; (b) $k(y)$ is generated by $R$ and the image of $B$; (c) the morphism
$\Spec(k(y))\to\Spec(R)\times\Spec(B)$ is a closed immersion.

(iv) Let $|\ |:B\to\Gamma\cup\{0\}$ be a valuation with kernel $y$, $\Gamma'\subseteq\Gamma$ be a convex
subgroup, and $R\subseteq R'$  be the valuation ring of $k(y)$ corresponding to the induced valuations
$k(y)\to\Gamma$ and $k(y)\to\Gamma\to\Gamma/\Gamma'$. Then the following conditions are equivalent: (a) there
exists a primary specialization $|\ |'$ corresponding to $\Gamma'$; (b) the image of $B$ in $k(y)$ is contained
in $R'$; (c) the morphism $y\to\Spec(B)$ extends to a morphism $\Spec(R')\to\Spec(B)$. Moreover, if the
conditions are satisfied then the kernel $y'$ of $|\ |'$ is the center of $R'$ on $\Spec(B)$. The equivalences
(a)$\Lra$(b) and (b)$\Lra$(c) are obvious. As for the additional claim, we note that the center of $R'$
corresponds to the kernel of the homomorphism $B\to R'\to R'/m_{R'}$, and the latter consists of the elements
$b\in B$ with $|b|\notin\Gamma'$, i.e. coincides with the kernel of $|\ |'$.
\end{rem}

Let, more generally, $\bfy=(y,R)$ be a valuation on a scheme $Y$. By a {\em primary specialization} of $\bfy$ we
mean a valuation $\obfy=(\oy,\oR)$ such that $\oy$ is a specialization of $y$ and the valuation $|\ |_\obfy$ on
$\calO_{Y,\oy}$ is a primary specialization of the valuation induced from $\bfy$ via the homomorphism
$\calO_{Y,\oy}\to\calO_{Y,y}$. Equivalently, if $\Spec(A)$ is an affine neighborhood of $\oy$ (and hence of $y$)
then the valuation induced by $(\oy,\oR)$ on $A$ is a primary specialization of the valuation induced by
$(y,R)$.

\begin{lem}\label{primlem}
Let $(y,R)$ be a valuation on a separated scheme $Y$.

(i) The set of primary valuations of $(y,R)$ is totally ordered by specialization;

(ii) If $Y$ is also quasi-compact then $(y,R)$ admits a minimal primary specialization.
\end{lem}
\begin{proof}
We claim that (i) follows from Remark \ref{primrem}(iv). Indeed, for any $R'$ with $R\subseteq R'\subseteq k(y)$
there exists at most one possibility to extend $y$ to a morphism $\Spec(R')\to Y$. So, if we have two primary
specializations $(y_1,R_1)$ and $(y_2,R_2)$ corresponding to valuation rings $R\subseteq R',R''\subseteq k(y)$,
then without loss of generality  we have that $R'\subseteq R''$ and the unique morphism $\Spec(R'')\to Y$ is
obtained by localizing the morphism $\Spec(R')\to Y$. Thus, $y_1$ is a specialization of $y_2$, and everything
reduces to affine theory of primary specializations on $\calO_{Y,y_1}$, see Remark \ref{primrem}(i). To prove
(ii) we note that $(y,R)$ admits a minimal primary specialization because if $\{(y_i,R_i)\}_{i\in I}$ denotes
the set of all primary specializations then the set of kernels $\{y_i\}_{i\in I}$ is totally ordered with
respect to specialization. By quasi-compactness there exists a point $\oy\in Y$ which is a specialization all
$y_i$'s. So, the claim reduces to the affine theory on $\calO_{Y,\oy}$, see Remark \ref{primrem}(ii).
\end{proof}

Finally, taking a morphism $f:Y\to X$ into account, by a {\em primary specialization} of an $X$-valuation
$\bfy=(y,R,\phi)$ we mean an $X$-valuation $\obfy=(\oy,\oR,\ophi)$ such that $(\oy,\oR)$ is a primary
specialization of $(y,R)$ and the image of $\ophi$ in $X$ is contained in the image of $\phi$ in $X$. Primary
specialization is a particular case of a specialization relation in $\Spa(Y,X)$. An $X$-valuation $(y,R,\phi)$
(resp. a valuation $(y,R)$) on $Y$ is called {\em minimal} if it has no non-trivial primary specializations.

\begin{lem}\label{primspecrem}
Let $(y,R,\phi)$ be an $X$-valuation on $Y$. Then any primary specialization $(\oy,\oR)$ of the valuation
$(y,R)$ admits at most one extension to a primary specialization $(\oy,\oR,\ophi)$ of $(y,R,\phi)$, and the
extension exists if and only if $f(\oy)$ belongs to the image of $\phi$. The latter is automatically the case
when $X$ is separated.
\end{lem}
\begin{proof}
Obviously, the assertion on $f(\oy)$ is necessary for an extension to exist. Furthermore, by Remark
\ref{primrem}(iv) there exists a valuation ring $R'$ with $R\subseteq R'\subseteq k(y)$ such that $y$ extends to
a morphism $\Spec(R')\to Y$ with $\oy$ being the image of the closed point. If $X$ is separated then the induced
map $\Spec(R')\to X$ must coincide with the corresponding localization of $\phi:\Spec(R)\to X$, hence we obtain
the last assertion of the lemma. The remaining claims are local at the center $x\in X$ of $R$ (i.e. the image of
the closed point of $\phi$). So, we can replace $X$ and $Y$ with a neighborhood of $x$ and its preimage
achieving that the schemes become separated. The uniqueness is now clear. To establish existence we should check
that the image of the homomorphism $\calO_{X,x}\to\calO_{Y,\oy}\to k(\oy)$ is in $\oR$. The latter follows from
the following two facts: (a) by existence of $\phi$ the image of $\calO_{X,x}$ in $k(y)$ is in $R$, (b) $\oR$ is
induced from $R$ in the sense that an element $\of\in\calO_{Y,\oy}$ satisfies $\of(\oy)\in\oR$ if and only if
$f(y)\in R$.
\end{proof}

The lemma shows that we can actually ignore $\phi$ when $X$ is separated. In particular, minimality of
$(y,R,\phi)$ is then equivalent to that of $(y,R)$.

\begin{cor}
\label{minvallem} Let $f:Y\to X$ be a separated morphism of qcqs schemes and $\bfy=(y,R,\phi)$ be an
$X$-valuation on $Y$. Then

i) the set of primary specializations of $\bfy$ is totally ordered and contains a minimal element,

(ii) $\bfy$ is minimal if and only if the morphism $h:\Spec(k(y))\to Y\times_X\Spec(R)$ is a closed immersion.
\end{cor}
\begin{proof}
The claim is local at the center of $x\in X$ of $R$ (with respect to $\phi$), hence we can assume that $X$ and,
hence, $Y$ are separated. Then primary specializations of $\bfy$ can be identified with primary specializations
of the valuation $(y,R)$, hence (i) follows from Lemma \ref{primlem}. To prove (ii) we note that as soon as $X$
is separated, $h$ is a closed immersion if and only if the morphism $\Spec(k(y))\to Y\times\Spec(R)$ is a closed
immersion. Hence the claim follows from Remark \ref{primrem}(iii).
\end{proof}

Until the end of \S\ref{nagchap}, we assume that $f:Y\to X$ is a separated morphism of qcqs schemes, unless the
contrary is said explicitly. We define the subset $\gtX=\Val_Y(X)\subset\ogtX$ as the set of all minimal
valuations and note that in view of Lemma \ref{minvallem} this agrees with the case studied in \ref{affsec}. We
do not introduce $\Val_Y(X)$ when $f$ is not separated: although the formal definition makes sense, it is not
clear if the obtained object is interesting. Note also that in affine situation such subsets were considered by
Huber, see \cite[2.6 and 2.7]{Hub1}. We provide $\gtX$ with the induced topology. The following lemma follows
easily from the valuative criterion of properness and Lemma \ref{afflem}.

\begin{lem}
\label{easylem} (i) If $X'$ is a $Y$-modification of $X$ then there are natural homeomorphisms
$\Spa(Y,X')\toisom\Spa(Y,X)$ and $\Val_Y(X')\toisom\Val_Y(X)$.

(ii) If $X'$ is an open subscheme of $X$ then its preimage in $\Val_Y(X)$ (resp. $\Spa(Y,X)$) is canonically
homeomorphic to $\Val_{Y'}(X')$ (resp. $\Spa(Y',X')$), where $Y'=X'\times_X Y$.
\end{lem}

\begin{rem}\label{basarem}
(i) If $f':Y'\to X'$ and $f:Y\to X$ are separated morphisms of qcqs schemes, and $g:f'\to f$ is a morphism such
that $g_Y$ is an open immersion and $g_X$ is separated and of finite type, then $\Spa(Y',X')$ maps
homeomorphically onto an open subspace of $\ogtX$. However, it may (and usually does) happen that the image of
$\Val_{Y'}(X')$ in $\ogtX$ is not contained in $\gtX$. The problem originates from the fact that a minimal
valuation on $Y'$ may admit non-trivial primary specializations on $Y$.

(ii) There exists a natural contraction $\pi_\gtX:\ogtX\to\gtX$ which maps any valuation to its minimal primary
specialization, but it is a difficult fact that $\pi_\gtX$ is continuous.

(iii) Using $\pi_\gtX$ we can extend $\Val$ to a functor by composing $\Spa(g)$ with the contraction $\pi_\gtX$
as $\Val(g):\gtX'\into\ogtX'\to\ogtX\to\gtX$. However, we do not know that it is continuous until continuity of
$\pi_\gtX$ is established.
\end{rem}

Actually, the above problems are closely related, and we will solve them only in the end of \S\ref{affinoidsec}.
Recall that if $X$ and $Y$ are qcqs then so are $\Spa(Y,X)$ and $\RZ_Y(X)$ (by Corollary \ref{qcqscor} and
\cite[3.3.1]{Temst}). Here is a partial (so far) result for $\Val_Y(X)$.

\begin{prop}
\label{qcprop} Assume that $f:Y\to X$ is a separated morphism of qcqs schemes. Then the spaces $\Val_Y(X)$ is
quasi-compact and the map $\psi:\Val_Y(X)\to\RZ_Y(X)$ is continuous.
\end{prop}
\begin{proof}
Let $\{\gtX_i\}_{i\in I}$ be an open covering of $\gtX$. Find open sets $\ogtX_i\subset\ogtX$ such that
$\gtX_i=\ogtX_i\cap\gtX$. Since any point of $\ogtX$ has a specialization in $\gtX$ by Corollary
\ref{minvallem}, $\{\ogtX_i\}_{i\in I}$ is a covering of $\ogtX$. By quasi-compactness of $\ogtX$, we can find a
subcovering $\{\ogtX_i\}_{i\in J}$ with a finite $J$, and then $\{\gtX_i\}_{i\in J}$ is a finite covering of
$\gtX$. Thus, $\gtX$ is quasi-compact.

We claim that for any $Y$-modification $X'\to X$, the map $\phi:\gtX\to X'$ is continuous. Indeed, if $U\subset
X'$ is open then its preimage in $\ogtX$ is the open subspace $\ogtX'\toisom\Spa(Y\times_{X'} U,U)$. Therefore,
the preimage of $U$ in $\gtX$ is the open set $\ogtX'\cap\gtX$, as required. Continuity of the maps $\phi$ (for
each $X'$) implies that the map $\psi:\gtX\to\RZ_Y(X)$ is continuous.
\end{proof}

\subsection{Valuative criteria}\label{valcritsec}

In the sequel, we will need to strengthen the classical valuative criteria of separatedness and properness,
\cite[$\rm II$, 7.2.3 and 7.3.8]{ega}. Our aim is to show that it suffices to consider valuative diagrams of
specific types. We say that a morphism is compatible with a commutative diagram, if the diagram remains
commutative after adjoining this morphism. Throughout \S\ref{valcritsec} $f$ is not assumed to be separated.

\begin{lem}
\label{valdiaglem1} Keep the notation of diagram (\ref{valdiag}) and set $K'=k(i(\eta))$ and $R'=R\cap K$. Then
diagram (\ref{valdiag}) completes uniquely to a commutative diagram
\begin{equation*}
\xymatrix{
\Spec(K) \ar[d]\ar[r] & \Spec(K') \ar[d]\ar[r] & Y\ar[d]\\
\Spec(R)\ar[r] & \Spec(R')\ar[r] & X}
\end{equation*}
and any morphism $h:\Spec(R)\to Y$ compatible with the above diagram is induced from a morphism $h':\Spec(R')\to
Y$ compatible with the diagram. In addition, $h$ determines $h'$ uniquely.
\end{lem}
\begin{proof}
The morphism $\Spec(K)\to Y$ obviously factors through $\Spec(K')$. The morphism $\Spec(R)\to X$ factors through
$\Spec(\calO_{X,x})$, where $x$ is the image of the closed point of $\Spec(R)$. The image of $\calO_{X,x}$ in
$R\subset K$ is contained in $K'$, hence the morphism $\Spec(R)\to X$ factors through $\Spec(R')$. By the same
reasoning, a morphism $h:\Spec(R)\to Y$ compatible with the diagram factors through $h':\Spec(R')\to Y$, and
they both are determined uniquely by the image of the closed point in $Y$.
\end{proof}

\begin{lem}
\label{valdiaglem2} Keep the notation of diagram (\ref{valdiag}) and assume that $R\subseteq R'\subseteq K$ is
such that the morphism $\Spec(R')\to X$ admits a lifting $g:\Spec(R')\to Y$ compatible with the diagram. Let
$\tilK$ be the residue field of $R'$ and $\tilR$ be the image of $R$ in $\tilK$.
\begin{equation*}
\xymatrix{
 & \Spec(K) \ar[d]\ar[r] & Y\ar[dd]^f\\
\Spec(\tilK)\ar[r]\ar[d] & \Spec(R')\ar[ru]^g\ar[d] & \\
\Spec(\tilR)\ar[r] & \Spec(R)\ar[r] & X}
\end{equation*}
Then any morphism $\tilh:\Spec(\tilR)\to Y$ compatible with the above diagram is induced from a morphism
$h:\Spec(R)\to Y$ compatible with the diagram and $\tilh$ determines $h$ uniquely.
\end{lem}
\begin{proof}
Consider a morphism $\tilh:\Spec(\tilR)\to Y$ compatible with the diagram. It suffices to show that it factors
through $\Spec(R)$, since the uniqueness is again trivial. Let $y$ be the image of the closed point of
$\Spec(\tilR)$, so $\tilh$ induces a homomorphism $\calO_{Y,y}\to\tilR$. Since $y$ is a specialization of the
image $y'$ of the closed point of $\Spec(R')$, we have also a homomorphism $\calO_{Y,y}\to\calO_{Y,y'}\to R'$.
Then the compatibility implies that the image of $\calO_{Y,y}$ in $\tilK=R'/m_{R'}$ lies in $\tilR$. Therefore,
the image of $\calO_{Y,y}$ in $R'$ lies in $R$ which is the preimage of $\tilR$ under $R'\to\tilK$, and we
obtain that the homomorphism $\calO_{Y,y}\to \tilR$ factors through $R$. It gives the desired morphism
$h:\Spec(R)\to Y$.
\end{proof}

Note that Lemma \ref{valdiaglem1} implies that it suffices to consider only the case when $k(i(\eta))\toisom K$
in the valuative criteria (i.e. it suffices to take valuative diagrams corresponding to the elements of
$\Spa(Y,X)$), and then Lemma \ref{valdiaglem2} and Remark \ref{primrem}(iv) imply that it even suffices to
consider only the valuative diagrams corresponding to the elements of $\Val_Y(X)$. It is also well known that in
the valuative criteria one can restrict to the case when the image of $\eta$ lies in a given dense subset which
is closed under generalization (e.g. the generic point of an irreducible scheme), and such strengthening is the
main issue of the following proposition.

\begin{prop}
\label{valcritprop} Assume that $h:Z\to Y$ and $f:Y\to X$ are morphisms of qcqs schemes and consider the natural
map $\opsi:\Spa(Z,Y)\to\Spa(Z,X)$.

(i) $f$ is separated if and only if $\opsi$ is injective.

(ii) Assume that $f$ is of finite type. Then $f$ is proper if and only if $\opsi$ is bijective.

(iii) If $f$ and $h$ are separated then $\opsi$ induces a map $\psi:\Val_Z(Y)\to\Val_Z(X)$, and $\opsi$ is
bijective if and only if $\psi$ is bijective.
\end{prop}
\begin{proof}
First we prove (iii). If $\bfz=(z,R,\phi_Y)$ is a point in $\Val_Z(Y)$ then the morphism $z\to\Spec(R)\times_YZ$
is a closed immersion. But the target is a closed subscheme in $\Spec(R)\times_XZ$ by separatedness of $f$, and
hence $\opsi(\bfz)$ is also a minimal valuation. Thus, $\opsi$ induces a map $\psi$ between the subsets $\Val$.
Next we relate the fibers of $\psi$ and $\opsi$. Consider any point $\bfz\in\Spa(Z,X)$ and let
$\bfz_0\in\Val_Z(X)$ be its minimal primary specialization. Then Lemma \ref{valdiaglem2} implies that the sets
$\opsi^{-1}(\bfz)$ and $\psi^{-1}(\bfz_0)$ are naturally bijective, and this proves (iii).

We will deal with (i) and (ii) simultaneously. The direct implications follow from the standard valuative
criteria. We will prove the opposite implications (which are refined valuative criteria) by getting a
contradiction. So, suppose that $f$ is not separated in (i), or of finite type, separated and not proper in (ii)
(if $f$ is not separated in (ii) then $\opsi$ cannot be bijective by (i)). By the standard valuative criterion
and Lemma \ref{valdiaglem1}, there exists an element $\bfy=(y,R_y,\phi_y)\in\Spa(Y,X)$ such that the number of
liftings of the morphism $\phi_y:\Spec(R_y)\to X$ to $Y$ is at least two in (i) or zero in (ii). Let $x$ denote
the center of $R_y$ on $X$.

By \cite[6.6.5]{egaI}, there exists a point $z\in Z$ for which $h(z)$ is a generalization of $y$, and so a
homomorphism $\calO_{Y,y}\to\calO_{Z,z}\to k(z)$ arises. Let $R'$ be any valuation ring of $k(z)$ which
dominates the image of $\calO_{Y,y}$. It gives rise to an element $(z,R',\phi')\in\Spa(Z,Y)$ centered on $y$.
Choose a valuation ring $\tilR$ of the residue field $\tilK$ of $R'$ such that $\tilR$ dominates the valuation
ring $R_y$ of $k(y)\subset\tilK$, and define a valuation ring $R$ of $k(z)$ as the composition of $R'$ and
$\tilR$. The compatible homomorphisms $\calO_{X,x}\to\calO_{Y,y}\to R'$ and $\calO_{X,x}\to R_y\to\tilR$ induce
a homomorphism $\calO_{X,x}\to R$, and we obtain the following commutative diagrams.
\begin{equation*}
\xymatrix{
 & \Spec(k(z)) \ar[d]\ar[r] & Z\ar[d] & \\
\Spec(\tilK)\ar[r]\ar[d] & \Spec(R')\ar[r]^{\phi'}\ar[d] & Y\ar[d] &
\Spec(\tilK)\ar[r]\ar[d] & \Spec(k(y))\ar[r]\ar[d] & Y\ar[d] \\
\Spec(\tilR)\ar[r] & \Spec(R)\ar[r]^{\phi_x} & X & \Spec(\tilR)\ar[r] & \Spec(R_y)\ar[r] & X}
\end{equation*}

Lemma \ref{valdiaglem1} implies that there is a one-to-one correspondence between morphisms $\Spec(R_y)\to Y$
and $\Spec(\tilR)\to Y$ compatible with the right diagram, and by Lemma \ref{valdiaglem2}, the latter morphisms
are in one-to-one correspondence with the morphisms $\phi:\Spec(R)\to Y$ compatible with the left diagram. So,
there are at least two such $\phi$'s in (i) and there is no such $\phi$ in (ii). Note that $\bfz=(z,R,\phi_x)$
is an element in $\Spa(Z,X)$, and any morphism $\phi$ as above gives a preimage of $\bfz$ in $\Spa(Z,Y)$. We
obtain that in the case (i), $\bfz$ has at least two preimages and so $\opsi$ is not injective. The same
argument would prove (ii) if we also know that, conversely, any preimage of $\bfz$ in $\Spa(Z,Y)$ comes from
$\phi$ as above. In other words, we want to show that any lift of $\phi_x$ to $\tilphi:\Spec(R)\to Y$ is
compatible with the whole left diagram, and this actually reduces to compatibility of $\tilphi$ with $\phi'$.
Note that $Y\to X$ is separated by the already established case (i), and the valuative criterion of
separatedness implies that the morphism $\phi'$ is uniquely determined by the morphisms $\Spec(k(z))\to Y$ and
$\Spec(R')\to X$. So, compatibility of $\tilphi$ with $\phi'$ is automatic.
\end{proof}

\subsection{Affinoid domains} \label{affinoidsec}

Let $f':Y'\to X'$ be another separated morphism of qcqs schemes and $g:f'\to f$ be a morphism. Recall that we
defined in \S\ref{spasec} a continuous map $\Spa(g):\ogtX'\to\ogtX$ which was shown to be injective if $g_Y$ is
an immersion and $g_X$ is separated. However, our definition of a map $\Val(g):\gtX'\to\gtX$ was rather
cumbersome because even if $\Spa(g)$ is injective, it does not have to respect the subspaces $\Val$ in the
spaces $\Spa$. The following proposition gives a criterion when $\Spa(g)$ does respect $\Val$'s.

\begin{prop}
\label{quasi-domprop} Suppose that $g_Y$ is an open immersion and $g_X$ is separated. Then
$\Spa(g)(\gtX')\subset\gtX$ if and only if the locally closed immersion $(g_Y,f'):Y'\to Y\times_X X'$ is a
closed immersion, in which case one actually has that $\gtX'=\Spa(g)^{-1}(\gtX)$.
\end{prop}
\begin{proof}
Suppose that $h:=(g_Y,f')$ is a closed immersion. Let $\bfy'=(y',R',\phi)\in\ogtX'$ be a point with
$\eta'=\Spec(k(y'))$ and $S'=\Spec(R')$, and let $\bfy=(y,R,\phi)$ be its image in $\ogtX$. By Lemma
\ref{minvallem}(ii), $\bfy'$ is minimal if and only if the natural morphism $\eta'\to Y'\times_{X'}S'$ is a
closed immersion. By closedness of $h$, the latter happens if and only if the composition morphism $\eta'\to
Y'\times_{X'}S'\to(Y\times_X X')\times_{X'}S'\toisom Y\times_X S'$ is a closed immersion. The latter happens if
and only if $\bfy$ is minimal because $k(y)\toisom k(y')$ and hence $R=R'\cap k(y)=R'$. Thus, under our
assumption on $h$, minimality of $\bfy'$ is equivalent to minimality of its image. This establishes the inverse
implication in the proposition, and the complement.

It remains to show that if $h$ is not a closed immersion then $\Spa(g)$ does not respect the subsets $\Val$.
Note that $h$ is a locally closed immersion because $g_Y$ is an open immersion, and assume that $h$ is not a
closed immersion. Set $Z=Y\times_X X'$ and find a $Z$-valuation $\bfy'=(y',R',\phi')$ of $Y'$ such that the
morphism $\phi':\Spec(R')\to Z$ cannot be lifted to a morphism $\Spec(R')\to Y'$. Replacing $\bfy'$ by its
minimal primary specialization, we achieve that $\bfy'$ is minimal and $R'\subsetneq k(y')$. Clearly $\bfy'$
defines an $X'$-valuation $\bfy=(y',R',\phi)$ on $Y'$ with $\phi=\pr_{X'}\circ\phi'$, and $\bfy$ is minimal
because any its non-trivial primary specialization corresponds to a lifting $\Spec(R'')\to Y'$ for some
$R'\subseteq R''\varsubsetneq k(y)$ and such a lifting would induce a lifting $\Spec(R'')\to Z$ corresponding to
a non-trivial primary specialization of $\bfy'$. Thus, $\bfy\in\gtX'$, but $\Spa(g)(\bfy)$ is not a minimal
$X$-valuation on $Y$ because the morphism $\Spec(R')\to X$ lifts to the morphism $\pr_Y\circ\phi:\Spec(R')\to
Y$.
\end{proof}

Let us assume that $g_Y$ is an open immersion and $g_X$ is separated and of finite type. We saw that if $h$ is a
closed immersion then $\gtX'$ is naturally identified with a quasi-compact open subset of $\gtX$ via $\Spa(g)$,
and we say in this case that $\gtX'$ is an {\em open subdomain} of $\gtX$. If, in addition, $X'$ and $Y'$ can be
chosen to be affine then we say that $\gtX'$ is an {\em affinoid subdomain} of $\gtX$. Note also that the
situation described in the proposition appears in Deligne's proof of Nagata compactification theorem under the
name of quasi-domination. (Recall that by a quasi-domination of $Y$ over $X'$ one means an open subscheme
$Y'\subset Y$ and a morphism $Y'\to X'$ such that the morphism $Y'\to Y\times_X X'$ is a closed immersion, see
\cite[\S2]{Con}.) The notion of quasi-domination plays a central role in Deligne's proof. We list simple
properties of open and affinoid subdomains in the following lemma and stress that it will be much more difficult
to prove that open subdomains are preserved under taking finite unions (in a sense, this is a typical situation
in algebraic geometry that preimages, intersections, projective limits, etc., are much easier for study than
pushouts, images, direct limits, etc.).

\begin{lem} \label{domlem}
Open subdomains are transitive and are preserved by finite intersections. Moreover, the intersection of open
subdomains $\Val_{Y_i}(X_i)$ with $i\in\{1,2\}$ is the open subdomain $\Val_{Y_1\cap Y_2}(X_1\times_X X_2)$. In
particular, if $X$ is separated and $\gtX_i$'s are affinoid then $\gtX_{12}$ is affinoid.
\end{lem}
\begin{proof}
This follows from the analogous Lemma \ref{afflem} concerning the spaces $\Spa$.
\end{proof}

The following remark will not be used in the sequel.

\begin{rem}
(i) Our definition of RZ spaces is a straightforward generalization of the classical one. It is also possible to
define RZ spaces directly as follows: an affinoid space is a topological space $\gtX=\Val_B(A)$ provided with
two sheaves of rings $\calO_\gtX\subset\calM_\gtX$ (which can be defined in a natural way), and general spaces
are pasted from affinoid ones along affinoid subdomains.

(ii) The following example illustrates a difference between adic and Riemann-Zariski spaces. Let $k$ be a field,
$A=B=A'=k[T]$, $B'=k[T,T^{-1}]$ and $\gtX,\gtX',\ogtX,\ogtX'$ are as above. Then $\ogtX'$ is a rational
subdomain in $\ogtX$ in the sense of \cite{Hub2}. From other side, $\gtX'$ is not an affinoid domain in $\gtX$.
Note that actually $(\gtX',\calO_{\gtX'})\toisom(\gtX,\calO_\gtX)\toisom X:=\Spec(A)$, but the sheaves
$\calM_\gtX$ and $\calM_{\gtX'}$ are not isomorphic at the point $x\in X$ with $T=0$. This can happen because
the local (and even a valuation) ring $\calO_{X,x}$ can be provided with two different structures of a
semi-valuation ring by choosing semi-fraction rings $\calM_{\gtX',x}=k(T)$ or $\calM_{\gtX,x}=\calO_{X,x}$. (See
also Remark \ref{lastrem}(ii).)
\end{rem}

\begin{theor}
\label{affbaseth} The affinoid subdomains of $\gtX$ form a basis of its topology.
\end{theor}
\begin{proof}
It follows from Lemma \ref{domlem} that we should prove that for any affine subset $\ogtX_0=\Spa(B_0,A_0)$ in
$\ogtX$ and a point $\bfy=(y,R,\phi)\in\gtX\cap\ogtX_0$ there exists an affinoid subdomain $\Val_\oY(\oX)$
containing $\bfy$ and contained in $\ogtX_0$. Moreover, we can assume that $X=\Spec(A)$ is affine because $\gtX$
is covered by open subdomains of the form $\Val_{Y'}(X')$, where $X'=\Spec(A)$ is an open subscheme of $X$ and
$Y'=X'\times_X Y$. In order to construct $\Val_\oY(\oX)$ as required we will extend diagram (\ref{twosqdiag}) to
the following one, where $\oY=\Spec(\oB)$ and $\oX=\Spec(\oA)$ will be finally defined in the end of the proof.
Recall that $\calO_\bfy$ is a semi-valuation ring with semi-fraction ring $\calO_{Y,y}$ and such that
$\calO_\bfy/m_y=R$.
$$
\xymatrix{
\Spec(k(y))\ar[r]\ar[d]& \Spec(\calO_{Y,y})\ar[r]\ar[d]& \oY\ar[r]\ar[d]& Y\ar[d]\\
\Spec(R)\ar[r]& \Spec(\calO_\bfy)\ar[r]& \oX\ar[r]& X }
$$

Since $\Spec(R)\times_X Y$ is closed in $\Spec(R)\times Y$ by separatedness of $X$, Lemma \ref{minvallem}(ii)
implies that the morphism $h:\Spec(k(y))\to\Spec(R)\times Y$ is a closed immersion. To explain the strategy of
the proof we remark that the morphism $\Spec(\calO_{Y,y})\to\Spec(\calO_\bfy)\times Y$ is a closed immersion
(actually it can be proved by the same argument as we use below), and our strategy will be to approximate
$\calO_\bfy$ and $\calO_{Y,y}$ by $A$-rings $\oA$ and $\oB$ so that $\oA$ is finitely generated over $A$,
$\oY=\Spec(\oB)$ is a neighborhood of $y$ and $\oY\to\oX\times Y$ is a closed immersion.

It will be more convenient to work with affine schemes and $Y$ is the only non-affine scheme in our
consideration, so let us cover $Y$ with open affine subschemes $Y_i=\Spec(B_i),Z_j=\Spec(C_j)$, where $1\le i\le
n$, $1\le j\le m$, $y\in Y_i$ and $y\notin Z_j$. Since $\Spec(B_0)$ contains $y$ by our assumptions, we also set
$Y_0=\Spec(B_0)$. For each $i$, $h$ factors through a closed immersion $\Spec(k(y))\to\Spec(R)\times Y_i$, hence
the images of $R$ and $B_i$ generate $k(y)$. Now, we will find a neighborhood $\oY=\Spec(\oB)$ of $y$ which is
contained in all $Y_i$'s and satisfies the following condition: for each $i$, $\oB$ is a localization of the
form $(B_i)_{f_i}$ and, the most important, we have that $f_i(y)\notin m_R$. Let us (until the end of this
paragraph only) call {\em $R$-localization} for localization of an affine neighborhood $\Spec(C)$ of $y$ at an
element $f$ such that $f(y)\notin m_R$. Obviously, $R$-localizations are transitive and we claim that the family
of $R$-localizations of each $Y_i$ form a basis of neighborhoods of $y$. Indeed, for any element $f\in B_i$ with
$f(y)\neq 0$ we can find $g\in B_i$ with $f(y)g(y)\notin m_R$ (we use that $B_i(y)$ generates $k(y)$ over $R$,
so it contains elements of arbitrary large valuation). Thus, $(B_i)_{fg}$ is an $R$-localization of $B_i$ where
$f$ is inverted and we obtain that the maximal (infinite) $R$-localization of $B_i$ is actually $\calO_{Y,y}$.
Now, set $\Spec(B)=\cap_{i=1}^n Y_i$ and find $R$-localizations $Y'_i=\Spec((B_i)_{g_i})$ contained in
$\Spec(B)$, and let $\oY=\Spec(\oB)$ be an $R$-localization of $\Spec(B)$ contained in all $Y'_i$. Then $\oY$ is
an $R$-localization of each $Y'_i$, hence an $R$-localization of each $Y_i$ too. So, $\oB=(B_i)_{f_i}$ is as
required.

Let $\oA$ be the preimage of $R$ under the character $\oB\to k(y)$ corresponding to $y$. Clearly $\oA$ contains
each element $f_i^{-1}$, hence the ring $\oB(y)=B_i(y)[f_i^{-1}(y)]$ is generated by $\oA(y)$ and $B_i(y)$. So,
we obtain epimorphisms $\oA\otimes B_i\to k(y)$, and then the homomorphisms $h_i:\oA\otimes B_i\to\oB$ are also
surjective because $\oA$ contains the kernel $p_y$ of $\oB\to k(y)$. In particular, each morphism
$\oY\to\oX\times Y_i$ is a closed immersion. We claim that actually, $\alp:\oY\to\oX\times Y$ is a closed
immersion, and to prove this we should check in addition that the morphisms $\alp_j:\oY\times_Y Z_j\to\oX\times
Z_j$ with $1\le j\le m$ are closed immersion. By separatedness of $Y$ the source is affine, hence $\oY\times_Y
Z_j=\Spec(\oC_j)$ where $\oC_j$ is generated by the images of $c_j:C_j\to\oC_j$ and $b_j:\oB\to\oC_j$. Since our
claim about $\alp$ would follow if we prove that the homomorphisms $h'_j:\oA\otimes C_j\to\oC_j$ are surjective,
it remains only to prove that for each $j$ the image of $h'_j$ contains the image of $b_j$. Since $y\in\oY$ and
$y\notin Z_j$ we have that $b_j(p_y)\oC_j=\oC_j$, and hence the equality $\oC_j=b_j(\oB)c_j(C_j)$ can be
strengthened as $\oC_j=b_j(p_y)c_j(C_j)$, i.e. $\oC_j$ is actually generated by $b_j(p_y)$ and $c_j(C_j)$. Since
$p_y\subset\oA$ by the definition of $\oA$, we obtain that $h'_j$ is onto, as claimed.

Now, the morphism $\oY\to\oX$ is almost as required: $\oY$ is open in $Y$ and $\alp$ is a closed immersion. In
addition, since $\bfy\subset\ogtX_0$, the image of $A_0$ under the homomorphism $A_0\to B_0\to\oB\to\oB(y)$ is
contained in $R$, and hence the image of $A_0$ in $\oB$ is actually contained in $\oA$. So, it only remains to
decrease the $A$-subalgebra $\oA\subset\oB$ so that $\oX=\Spec(\oA)$ becomes of finite type over $X$ but all
good properties are preserved: $\alp$ is still a closed immersion, and $\oA$ contains the image of $A_0$ in
$\oB$. As we saw, $\alp$ being a closed immersion is equivalent to surjectivity of the homomorphisms
$h_i:\oA\otimes B_i\to\oB$ and $h'_j:\oA\otimes C_j\to\oC_j$. Since the homomorphisms $B_i\to\oB$ and
$C_j\to\oC_j$ are of finite type, all we need for surjectivity of $h_i$'s and $h'_j$'s is a finite subset
$S\subset\oA$. So, replacing $\oA$ with its $A_0$-subalgebra generated by $S$ we obtain $\oX$ as required.
Obviously, $\Val_\oY(\oX)$ is an affinoid domain containing $\bfy$, and $\Val_\oY(\oX)$ is contained in
$\ogtX_0$ because $\oY$ is an open subscheme in $Y_0$ and the morphism $\oY\to X_0$ (obtained as $\oY\to Y_0\to
X_0$) factors through $\oX$.
\end{proof}

\begin{cor}\label{qccor}
The space $\gtX$ is qcqs.
\end{cor}
\begin{proof}
Any open subdomain is quasi-compact by Proposition \ref{qcprop}, and their intersection is quasi-compact by
Lemma \ref{domlem}. Since open subdomains generate the topology of $\gtX$ by Theorem \ref{affbaseth} we obtain
the corollary.
\end{proof}

Recall that we defined in Remark \ref{basarem} the contraction $\pi_\gtX:\ogtX\to\gtX$ and used it to define the
maps $\Val(g):\gtX'\to\gtX$ for $g:f'\to f$.

\begin{cor}
The contraction $\pi_\gtX$ is continuous. In particular, the maps $\Val(g)$ are continuous.
\end{cor}
\begin{proof}
Since open subdomains $\gtX'=\Val_{Y'}(X')$ form a basis of the topology of $\gtX$ by Theorem \ref{affbaseth},
it suffices to prove that the preimage of $\gtX'$ in $\Spa(Y,X)$ is open. Since the minimality condition in
$\Spa(Y,X)$ and $\Spa(Y',X')$ agree, $\pi^{-1}(\gtX')$ coincides with the open affine subset $\Spa(Y',X')$.
\end{proof}

\subsection{$Y$-blow ups of $X$} \label{blowupsec}

In this section we assume that $f$ is affine. Then we will show that there exists a large family of projective
$Y$-modifications of $X$ having good functorial properties. Using these morphisms we will be able to describe
the set $\Val_Y(X)$ very concretely. Since the results of \S\ref{blowupsec} are inspired in part by Raynaud's
theory of formal models, we will sometimes indicate similarity between our results and Raynaud's theory by
referencing to \cite{BL}.

\begin{defin}\label{blowdef}
A $Y$-modification $g_i:X_i\to X$ is called a {\em $Y$-blow up of $X$} if there exists a $g_i$-ample
$\calO_{X_i}$-module $\calL$ provided with a homomorphism $\ve:\calO_{X_i}\to\calL$ such that
$f_i^*(\ve):\calO_Y\toisom f_i^*(\calL)$.  We call $\ve$ a {\em $Y$-trivialization} of $\calL$; actually it is a
section of $\calL$ that is invertible on the image of $Y$.
\end{defin}

It will be more convenient to say $X$-ample instead of $g_i$-ample in the sequel.

\begin{lem}
\label{blowuplem} The $Y$-blow ups satisfy the following properties.

(i) Suppose that $X_j\to X_i$ and $X_i\to X$ are $Y$-modifications such that $X_j$ is a $Y$-blow up of $X$. Then
$X_j$ is a $Y$-blow up of $X_i$.

(ii) The family of $Y$-blow ups of $X$ is filtered.

(iii) The composition of $Y$-blow ups $g_{ij}:X_j\to X_i$ and $g_i:X_i\to X$ is a $Y$-blow up.
\end{lem}
\begin{proof}
The first statement is obvious because any $X$-ample $\calO_{X_j}$-module $\calL$ is $X_i$-ample, and the notion
of $Y$-trivialization of $\calL$ depends only on the morphism $f_j:Y\to X_j$.

(ii) Let $X_i,X_j$ be two $Y$-blow ups of $X$. Find $X$-ample sheaves $\calL_i,\calL_j$ with $Y$-trivializations
$\ve_i,\ve_j$. Then the $X$-proper scheme $X_{ij}=X_i\times_X X_j$ possesses an $X$-ample sheaf
$\calL=p_i^*(\calL_1)\otimes p_j^*(\calL_2)$, where $p_i,p_j$ are the projections. The natural isomorphism
$\calO_{X_{ij}}\toisom\calO_{X_{ij}}\otimes\calO_{X_{ij}}$ followed by $f_i^*(\ve_i)\otimes
f_j^*(\ve_j):\calO_{X_{ij}}\otimes\calO_{X_{ij}}\to\calL$ provides a $Y$-trivialization of $\calL$. Consider the
scheme-theoretic image $X'$ of $Y$ in $X_{ij}$, and let $\calL'$ and $\ve'$ be the pull backs of $\calL$ and
$\ve$. Then $(X',\calL',\ve')$ is a $Y$-blow up of $X$ which dominates $X_i$ and $X_j$.

(iii) Choose an $X$-ample $\calO_{X_i}$-sheaf $\calL_i$ and an $X_i$-ample $\calO_{X_j}$-sheaf $\calL_j$ with
$Y$-trivializations $\ve_i$ and $\ve_j$. By \cite[$\rm II$, 4.6.13(ii)]{ega}, the sheaf $\calL_j\otimes
g_{ij}^*(\calL_i^{\otimes n})$ is $X$-ample for sufficiently large $n$. It remains to notice that the
composition of $\calO_{X_j}\toisom\calO_{X_j}\otimes\calO_{X_j}^{\otimes n}$ with $\ve_j\otimes
g_{ij}^*(\ve_i^{\otimes n})$ is a $Y$-trivialization.
\end{proof}

We will need an explicit description of $Y$-blow ups. Let $\calE\subset f_*(\calO_Y)$ be a finitely generated
$\calO_X$-submodule containing the image of $\calO_X$, and let $\calE^n\subset f_*(\calO_Y)$ denote the
$\calO_X$-modules which are powers of $\calE$ with respect to the natural multiplication on $f_*(\calO_Y)$ (so
$\calE^0$ is the image of $\calO_X$). We claim that $X_\calE:=\bfProj(\oplus_{n=0}^\infty\calE^n)$ is a
$Y$-modification of $X$. Clearly, $X_\calE$ is $X$-projective and there is a natural morphism
$g_\calE:Y=\bfSpec(f_*(\calO_Y))\to\bfSpec(\cup_{n=0}^\infty\calE^n)$ where the union is taken inside
$f_*(\calO_Y)$. The target of $g_\calE$ is the $X$-affine chart of $X_\calE$ defined by non-vanishing of the
section $s\in\Gamma(\calE)$ which comes from the unit section of $\calO_X$, in particular, a map $Y\to X_\calE$
naturally arises. In addition, the very ample sheaf $\calO_{X_\calE}(1)$ on $X_\calE$ has a $Y$-trivialization
$\calO_{X_\calE}\to\calO_{X_\calE}(1)$ induced by $s$. So, among all properties of $Y$-blow ups it remains to
check that $g_\calE$ is schematically dominant. The latter can be checked locally over $X$, so assume that
$X=\Spec(A)$, $Y=\Spec(B)$ and $E\subset B$ is an $A$-module containing $1$. Then $X_E=\Proj(\oplus_{n=0}^\infty
E^n)$ is glued from affine charts $(X_E)_b$ given by non-vanishing of elements $b\in E$, so it suffices to show
that the morphism $\alp:Y\times_{X_E}(X_E)_b\to(X_E)_b$ is schematically dominant. Note that the source is the
localization of $Y$ at $b$, and so it is isomorphic to $\Spec(B_b)$, and the target is $\Spec(C)$ where $C$ is
the zeroth graded component of $(\oplus_{n=0}^\infty E^n)_b$. But $C=\injlim_n b^{-n}(E^n/I_n)$, where $I_n$ is
the submodule of elements killed by a power of $b$, and the kernel of the homomorphism $E^n\into B\to B_b$ is
$I_n$. Hence $b^{-n}(E^n/I_n)\into B_b$ and therefore $C\into B$. In particular, $\alp$ is schematically
dominant.

\begin{lem}
\label{expblowuplem} Any $Y$-blow up of $X$ is isomorphic to some $X_\calE$ as a $Y$-blow up of $X$.
\end{lem}
\begin{proof}
Let $g_i:X_i\to X$ be a $Y$-blow up. Find an $X$-ample $\calO_{X_i}$-module $\calL$ with a $Y$-trivialization
$\ve:\calO_{X_i}\to\calL$. Then there is a closed immersion of $X$-schemes $h:X_i\to
P:=\bfProj(\oplus_{n=0}^\infty(g_i)_*\calL^{\otimes n})$ and the morphism $h\circ f_i:Y\to X_i\to P$ factors
through the chart of $P$ given by non-vanishing of the section $s\in\Gamma((g_i)_*\calL)$ corresponding to
$\ve$. The latter chart is of the form $\bfSpec(\calA)$ where $\calA$ is the zeroth graded component of the
localization $(\oplus_{n=0}^\infty(g_i)_*\calL^{\otimes n})_s$. Composing the $\calO_X$-homomorphism
$(g_i)_*\calL\to\calA$ that takes $u$ to $s^{-1}u$ with the $\calO_X$-homomorphism $\calA\to f_*(\calO_Y)$
corresponding to $f_i$ we obtain a homomorphism $(g_i)_*\calL\to f_*(\calO_Y)$ that takes $s$ to the unit
section. Now we can define $\calE$ to be the image of $(g_i)_*\calL$ in $f_*(\calO_Y)$, and we claim that
actually $X_i\toisom X_\calE$ as a $Y$-modification of $X$. Indeed, the obvious epimorphism
$\oplus_{n=0}^\infty(g_i)_*\calL^{\otimes n}\to\oplus_{n=0}^\infty\calE^n$ corresponds to a closed immersion
$X_\calE\to P$ which agrees with the morphisms $Y\to X_\calE$ and $Y\to P$. Since, the first morphism is
schematically dominant, $X_\calE$ is the schematic image of $Y$ in $P$, hence it must coincide with $X_i$ as the
closed subscheme of $P$.
\end{proof}

\begin{cor}
\label{extblowuplem} Assume that $X'$ is an open subscheme of $X$ and $Y'=f^{-1}(X')$. Then any $Y'$-blow up
$X'_i\to X'$ extends to a $Y$-blow up $X_i\to X$.
\end{cor}
\begin{proof}
Let $f':Y'\to X'$ be the restriction of $f$, so $f'_*(\calO_{Y'})$ is the restriction of $f_*(\calO_Y)$ on $X'$.
By the lemma, a $Y'$-blow up of $X'$ is determined by a finitely generated $\calO_{X'}$-submodule $\calE'\subset
f'_*(\calO_{Y'})$ containing the image of $\calO_{X'}$. By \cite[6.9.7]{egaI}, one can extend $\calE'$ to a
finitely generated $\calO_X$-submodule $\calE\subset f_*(\calO_Y)$. Replacing $\calE$ by $\calE+\calO_X$, if
necessary, we can achieve that $\calE$ contains the image of $\calO_X$. Now, $\calE$ defines a required
extension of the blow up.
\end{proof}

\begin{rem}
(i) Lemma \ref{expblowuplem} indicates that the notion of $Y$-blow up is in some sense a generalization of the
notion of $U$-admissible blow up, where $i:U\into X$ is a schematically dense open subscheme, to the case of an
arbitrary affine morphism $Y\to X$. Indeed, there is much similarity, but the notions are not equivalent in
general: both $U$-admissible blow ups and $U$-blow ups are of the form $\Proj(\oplus_{n=0}^\infty\calE^n)$, but
in the first case $\calE$ is an $\calO_X$-submodule of $\calO_X$ which is trivial over $U$, and in the second
one $\calE$ is an $\calO_X$-submodule of $i_*(\calO_U)$ that contains $\calO_X$ (so, it is trivial over $U$ as
well). The important case when these notions agree was pointed out by the referee: it follows from \cite[$\rm
II$, 3.1.8(iii)]{ega} that $U$-admissible blow ups and $U$-blow ups agree when $X\setminus U$ is the zero set of
an invertible sheaf of ideals.

(ii) Basic facts concerning compositions, extensions, etc., (see the above lemmas) hold for both families of
$U$-modifications, but a slight advantage of $U$-blow ups is that the proofs seem to be easier. For example,
compare with \cite[1.2]{Con} where one proves that $U$-admissible blow ups are preserved by compositions.
\end{rem}

The following lemma is an analog of \cite[4.4]{BL}.

\begin{lem}
\label{submodlem} Given a quasi-compact open subset $\gtU\subset\gtX=\Val_Y(X)$, there exists a $Y$-modification
$X'\to X$ and an open subscheme $U\subset X'$ such that $\gtU$ is the preimage of $U$ in $\gtX$.
\end{lem}
\begin{proof}
If $X_1\. X_n$ form a finite open affine covering of $X$ and $Y_i=f^{-1}(X_i)$ then $\gtX_i=\Val_{Y_i}(X_i)$
form an open covering of $\gtX$ by Lemma \ref{easylem}. It suffices to separately solve our problem for each
$\gtX_i$ with $\gtU_i:=\gtU\cap\gtX_i$ because any $Y_i$-blow up of $X_i$ extends to a $Y$-blow up of $X$, and
$Y$-blow ups of $X$ form a filtered family. Thus, we can assume that $X=\Spec(A)$, and then $Y=\Spec(B)$. We can
furthermore assume that $\gtU=\gtX\cap\Spa(B_b,A[a_1/b\. a_n/b])$ with $a_i,b\in B$ because as we saw in the
proof of Lemma \ref{spahomlem}, the sets $\Spa(B_b,A[a_1/b\. a_n/b])$ form a basis of the topology of
$\Spa(B,A)$. Now, the morphism $Y\to\Proj(A[T_1,T_{a_1}\. T_{a_n},T_b])$ defined by $(1,a_1\. a_n,b)$ determines
a required $Y$-blow up $X'\to X$ with $U$ given by the condition $T_b\neq 0$.
\end{proof}

\begin{cor}
\label{homeomcor} The map $\psi:\Val_Y(X)\to\RZ_Y(X)$ is a homeomorphism.
\end{cor}
\begin{proof}
Recall that $\psi$ is surjective and continuous by Propositions \ref{goodcaseprop} and \ref{qcprop},
respectively. From other side, the lemma implies that $\psi$ is injective and open. Indeed, any open
quasi-compact $\gtU\subset\gtX$ is the full preimage of some $U\subset X'$ for a $Y$-modification $X'\to X$,
hence $\psi(\gtU)$, which is the full preimage of $U$ in $\RZ_Y(X)$, is open. In addition, since any pair of
different points of $\gtX$ is distinguished by some open quasi-compact set $\gtU\subset\gtX$, their images in an
appropriate $X'$ do not coincide.
\end{proof}

We use the corollary to identify $\gtX$ with $\RZ_Y(X)$ when $f$ is decomposable. In particular, this provides
$\gtX$ with a sheaf $\calO_\gtX$ of regular functions which was earlier defined on $\RZ_Y(X)$, and for any point
$\bfx\in\gtX$, thanks to Proposition \ref{goodcaseprop}, the semi-valuation ring $\calO_\bfx$ obtains a new
interpretation as the stalk of $\calO_\gtX$ at $\bfx$. As another corollary of Lemma \ref{submodlem} we obtain
the following version of Chow lemma.

\begin{cor}
\label{Chowcor} Any $Y$-modification $\oX\to X$ is dominated by a $Y$-blow up of $X$.
\end{cor}
\begin{proof}
Let $\oU_1\.\oU_n$ be an affine covering of $\oX$, and let $Y_i$ and $\gtU_i$ denote the preimages of $\oU_i$ in
$Y$ and $\gtX$, respectively.  By Lemma \ref{submodlem}, we can find a $Y$-blow up $X'\to X$ and a covering
$\{U'_i\}$ of $X'$, whose preimage in $\gtX$ coincides with $\{\gtU_i\}$. Note that the scheme-theoretic image
$X''$ of $Y$ in $\oX\times_X X'$ is a $Y$-modification of both $X'$ and $\oX$. So, it suffices to show that
$X''$ is a $Y$-blow up of $X$.

Since the preimages of $\oU_i$ and $U'_i$ in $\gtX$ coincide, their preimages in $X''$ coincide too, and we will
denote them as $U''_i\into X''$. Consider the induced $Y$-modification $h:X''\to X'$ with restrictions
$h_i:U''_i\to U'_i$. For any $1\le i\le n$, the proper morphism $h_i$ is affine because the morphism $\oU_i\to
X$ is affine and $U''_i$ is closed in $U'_i\times_X\oU_i$. Thus, $h_i$ is finite, and therefore $h$ is finite.
We claim that finiteness of $h$ implies that it is a $Y$-blow up (this claim is an analog of \cite[4.5]{BL}).
Indeed, $\calO_{X''}$ is very ample relatively to $h$ because $h$ is affine, and the identity homomorphism gives
its $Y$-trivialization. Thus, $X''$ is a $Y$-blow up of $X$ by Lemma \ref{blowuplem}(iii).
\end{proof}

\subsection{Decomposable morphisms}
\label{mainsec}

In this section we will complete a basic description of the relative Riemann-Zariski space $\gtX$ associated
with a separated morphism $f:Y\to X$ between qcqs schemes by proving that the finite union of open domains is an
open domain, and any open domain in $\gtX$ is of the form $\Val_\oY(\oX)$ where the morphism $\oY\to\oX$ is
affine and schematically dominant. The first claim actually means that any quasi-compact open subset is an open
domain, i.e. admits a model by a morphism of schemes, and the second claim states that this model can be chosen
to be affine. In particular, applying the second claim to $\gtX$ itself we obtain a bijection
$\Val_{\oY}(\oX)\toisom\Val_Y(X)$ with $\oY=Y$ and affine morphism $\oY\to\oX$. But then $\oX$ is proper over
$X$ by the valuative criterion \ref{valcritprop}, and hence $X$ admits a $Y$-modification $\oX$ such that the
morphism $Y\to\oX$ is affine. Thus, the morphism $f:Y\to X$ is decomposable and this gives a new proof of
Theorem \ref{decompth}. In particular, one obtains new proofs of Nagata compactification and Thomason
approximation theorems.

\begin{theor}\label{domth}
Let $f:Y\to X$ be a separated morphism between qcqs schemes and $\gtX=\Val_Y(X)$. Then

(i) open domains in $\gtX$ are closed under finite unions,

(ii) any open domain $\gtX'$ is of the form $\Val_\oY(\oX)$, where the morphism $\oY\to\oX$ is affine and
schematically dominant.
\end{theor}
\begin{proof}
Note that any affinoid domain satisfies the assertion of (ii) (since schematical dominance is achieved by simply
replacing $\oX$ with the schematic image of $\oY$), and by Theorem \ref{affbaseth} and Corollary \ref{qccor},
$\gtX'$ admits a finite affinoid covering. Therefore, both (i) and (ii) would follow if we prove the following
claim: the union of two domains satisfying the assertion of (ii) is an open domain that satisfies the assertion
of (ii). So, we assume that $\gtX'=\gtX_1\cup\gtX_2$ where $\gtX_i=\Val_{Y_i}(X_i)$ with $i\in\{1,2\}$ are open
subdomains with affine morphisms $Y_i\to X_i$.

Set $\gtX_{12}=\gtX_1\cap\gtX_2$ and $Y_{12}=Y_1\cap Y_2$. In the sequel, we will act as in Step 3 of the proof
of Theorem \ref{approxtheor}, and the main difference is that we will use $Y_i$-blow ups instead of affine
morphisms. For reader's convenience, we provide a commutative diagram containing the main objects which were and
will be introduced.

$$
\xymatrix{
Y_1\ar[d]  & & Y_{12} \ar[d]\ar@{_{(}->}[ll]\ar@{^{(}->}[rr] & & Y_2\ar[d]\\
\gtX_1\ar[d]  & & \gtX_{12} \ar[d]\ar@{_{(}->}[ll]\ar@{^{(}->}[rr] & & \gtX_2\ar[d]\\
Z_1\ar[d]  & & Z_{12} \ar[ld]\ar[rd]\ar@{_{(}->}[ll]\ar@{^{(}->}[rr] & & Z_2\ar[d]\\
X_1 & X'_1\ar@{_{(}->}[l]& & X'_2\ar@{^{(}->}[r] & X_2}
$$

Since $Y_i$'s are $X_i$-affine, Lemma \ref{submodlem} implies that we can replace $X_i$'s by their $Y_i$-blow
ups such that each $X_i$ contains an open subscheme $X'_i$, whose preimage in $\gtX_i$ coincides with
$\gtX_{12}$. Then the preimage of $X'_i$ in $Y$ is, obviously, $Y_{12}$. It can be impossible to glue $X_i$'s
along $X'_i$'s, but by Lemma \ref{easylem}(ii), we at least know that $\Val_{Y_{12}}(X'_i)\toisom\gtX_{12}$ for
$i=1,2$. Let $T$ be the scheme-theoretic image of $Y_{12}$ in $X'_1\times_X X'_2$; it is obviously separated
over $X'_i$'s. Moreover, $\Val_{Y_{12}}(T)\toisom\Val_{Y_{12}}(X'_1)\cap\Val_{Y_{12}}(X'_2)=\gtX_{12}$ by Lemma
\ref{domlem}, and, therefore, $T$ is a $Y_{12}$-modification of $X'_i$'s by the valuative criterion
\ref{valcritprop}.

By Corollary \ref{Chowcor}, we can find a $Y_{12}$-blow up $T'\to X'_1$, which dominates $T$. It still can
happen that $T'$ is not a $Y_{12}$-blow up of $X'_2$, but it is dominated by a $Y_{12}$-blow up $Z_{12}\to
X'_2$. Then $Z_{12}\to T'$ is a $Y_{12}$-blow up by Lemma \ref{blowuplem}(i), and hence $Z_{12}\to X'_1$ is a
$Y_{12}$-blow up by Lemma \ref{blowuplem}(iii). By Lemma \ref{extblowuplem}, we can extend the $Y_{12}$-blow ups
$Z_{12}\to X'_i$ to $Y_i$-blow ups $Z_i\to X_i$. Then, the finite type $X$-schemes $Z_i$ can be glued along the
subschemes $X$-isomorphic to $Z_{12}$ to a single $X$-scheme $\oX$ of finite type, and the schematically
dominant affine morphisms $Y_i\to Z_i$ glue to a single schematically dominant affine morphism $\oY\to\oX$. Note
that $\Val_{Y_i}(Z_i)=\gtX_i$ is the preimage of $Z_i$ in $\Val_\oY(\oX)$, in particular, the latter is covered
by its open subdomains $\gtX_i$, $i\in\{1,2\}$. Now, it remains to show that $\Val_\oY(\oX)$ is an open
subdomain in $\gtX$, since this would immediately  imply that $\Val_\oY(\oX)$ is a required model of $\ogtX$.
The morphism $\alp:\oY\to\oX\times_X Y$ is glued from the morphisms $\alp_i:Y_i\to Z_i\times_X Y$ because $Y_i$
is the preimage of $Z_i$ in $Y$, but $\alp_i$'s are closed immersions by the construction. So, $\alp$ is a
closed immersion as well, and we are done.
\end{proof}

\begin{cor}\label{lastcor}
The map $\eta:Y\to\gtX:=\RZ_Y(X)$ is injective, each point $\bfx\in\RZ_Y(X)$ possesses a unique minimal
generalization $y$ in $\eta(Y)$, $\calM_{\gtX,\bfx}\toisom\calO_{Y,y}$, and the stalk $\calM_{\gtX,\bfx}$ is the
semi-fraction ring of the semi-valuation ring $\calO_{\gtX,\bfx}$. In particular, $\calO_\gtX$ is a subsheaf of
$\calM_\gtX$.
\end{cor}
\begin{proof}
By Theorem \ref{domth} and Corollary \ref{homeomcor}, we can identify $\gtX$ with $\Val_Y(X)$. So, a point
$\bfx$ can be interpreted as an $X$-valuation $(y,R,\phi)$ on $Y$. Then it is clear that the map $\eta$ sends
$y\in Y$ to a trivial valuation $(y,k(y),f|_y)$ (with the obvious morphism $f|_y:\Spec(k(y))\to X$), and for an
arbitrary $\bfx=(y,R,\phi)$ its minimal generalization in $\eta(Y)$ is $(y,k(y),f|_y)$. Uniqueness of minimal
generalization implies that the stalk of $\calM_\gtX=\eta_*(\calO_Y)$ at $\bfx$ is simply $\calO_{Y,y}$, so it
remains to recall that the latter is the semi-fraction field of the semi-valuation ring $\calO_\bfx$ defined in
\S\ref{firstsec}, which coincides with the stalk $\calO_{\gtX,\bfx}$ by Proposition \ref{goodcaseprop}.
\end{proof}

\end{document}